\documentclass[reqno]{amsart}

\usepackage{amsmath}
\usepackage{amssymb}
\usepackage{amsfonts}
\usepackage{amsthm}

\usepackage[utf8]{inputenc}
\usepackage[T1]{fontenc}

\usepackage{dsfont}

\usepackage{algorithmic,algorithm}
\usepackage{algorithm2e_extension}
\SetKwBlock{Repeat}{Repeat}{}

\usepackage[dvipsnames,svgnames]{xcolor}

\usepackage{subfigure}
\usepackage{tikz}

\usepackage{acronym}
\usepackage{paralist}

\usepackage[numbers,sort&compress]{natbib}

\usepackage{hyperref}
\hypersetup{
colorlinks=true,
linktocpage=true,
pdfstartview=FitH,
breaklinks=true,
pdfpagemode=UseNone,
pageanchor=true,
pdfpagemode=UseOutlines,
plainpages=false,
bookmarksnumbered,
bookmarksopen=false,
bookmarksopenlevel=1,
hypertexnames=true,
pdfhighlight=/O,
urlcolor=Maroon,linkcolor=RoyalBlue!70!black,citecolor=DarkGreen!70!black, 
pdftitle={},
pdfauthor={},
pdfsubject={},
pdfkeywords={},
pdfcreator={pdfLaTeX},
pdfproducer={LaTeX with hyperref}
}

\newcommand{\bA}{\mathbf{A}}
\newcommand{\bB}{\mathbf{B}}

\newcommand{\bH}{\mathbf{H}}
\newcommand{\bI}{\mathbf{I}}

\newcommand{\bM}{\mathbf{M}}

\newcommand{\bO}{\mathbf{O}}
\newcommand{\bP}{\mathbf{P}}
\newcommand{\bQ}{\mathbf{Q}}

\newcommand{\bU}{\mathbf{U}}
\newcommand{\bV}{\mathbf{V}}
\newcommand{\bW}{\mathbf{W}}
\newcommand{\bX}{\mathbf{X}}
\newcommand{\bY}{\mathbf{Y}}

\newcommand{\bx}{\mathbf{x}}
\newcommand{\by}{\mathbf{y}}
\newcommand{\bz}{\mathbf{z}}

\newcommand{\C}{\mathbb{C}}
\newcommand{\R}{\mathbb{R}}

\newcommand{\N}{\mathbb{N}}

\DeclareMathOperator{\argdot}{\cdotp}

\DeclareMathOperator{\bigoh}{\mathcal O}
\DeclareMathOperator{\card}{card}

\DeclareMathOperator{\ex}{\mathbb{E}}
\DeclareMathOperator{\given}{\mid}
\DeclareMathOperator{\grad}{\nabla}

\DeclareMathOperator{\one}{\mathds{1}}

\DeclareMathOperator{\tr}{tr}

\DeclareMathOperator{\eff}{eff}

\newcommand{\dd}{\:d}

\newcommand{\eps}{\varepsilon}

\newcommand{\filter}{\mathcal{F}}
\newcommand{\from}{\colon}

\newcommand{\herm}{\boldsymbol{\mathcal{H}}}

\newcommand{\mg}{\succ}
\newcommand{\mgeq}{\succcurlyeq}

\newcommand{\abs}[1]{\left\lvert #1 \right\rvert}
\newcommand{\smallabs}[1]{\lvert #1 \rvert}

\newcommand{\norm}[1]{\left\| #1 \right\|}
\newcommand{\smallnorm}[1]{\| #1 \|}

\newcommand{\dis}{\displaystyle}
\newcommand{\txs}{\textstyle}

\newcommand{\insum}{\sum\nolimits}

\theoremstyle{plain}
\newtheorem{theorem}{Theorem}
\newtheorem{corollary}[theorem]{Corollary}
\newtheorem*{corollary*}{Corollary}
\newtheorem{lemma}[theorem]{Lemma}
\newtheorem{proposition}[theorem]{Proposition}

\theoremstyle{definition}

\newtheorem*{definition*}{Definition}
\newtheorem{assumption}{Assumption}

\theoremstyle{remark}
\newtheorem{remark}{Remark}
\newtheorem*{remark*}{Remark}

\numberwithin{equation}{section}
\numberwithin{theorem}{section}
\numberwithin{remark}{section}
\numberwithin{example}{section}

\newcommand{\play}{\mathcal{K}}

\newcommand{\pay}{u}

\newcommand{\obj}{f}
\newcommand{\objmean}{F}
\newcommand{\strat}{\boldsymbol{\mathcal{X}}}

\newcommand{\eq}{\bX^{\ast}}

\newcommand{\temp}{\tau}
\newcommand{\discrate}{r}
\newcommand{\step}{\gamma}

\newcommand{\vbound}{V}
\newcommand{\pf}{\mathcal{Z}}

\newcommand{\tx}{M}
\newcommand{\rx}{N}
\newcommand{\rate}{R}
\newcommand{\pmax}{P}
\newcommand{\pot}{\objmean}

\begin{document}


\title
[Stochastic Semidefinite Programming via Stochastic Approximation]
{A Stochastic Approximation Algorithm for Stochastic Semidefinite Programming}

\author[B.~Gaujal and P. ~Mertikopoulos]{Bruno Gaujal and Panayotis Mertikopoulos}

\address[B.~Gaujal]
{Inria\\
and
Univ. Grenoble Alpes, LIG, F-38000 Grenoble, France}
\email{\href{mailto:bruno.gaujal@inria.fr}{bruno.gaujal@inria.fr}}
\urladdr{\url{http://mescal.imag.fr/membres/bruno.gaujal}}

\address[P.~Mertikopoulos]
{CNRS (French National Center for Scientific Research), LIG, F-38000 Grenoble, France\\
and
Univ. Grenoble Alpes, LIG, F-38000 Grenoble, France}
\email{\href{mailto:panayotis.mertikopoulos@imag.fr}{panayotis.mertikopoulos@imag.fr}}
\urladdr{\url{http://mescal.imag.fr/membres/panayotis.mertikopoulos}}


\newacro{KKT}{Karush\textendash Kuhn\textendash Tucker}
\newacro{MSE}{mean squared error}
\newacro{AIMD}{additive increase, multiplicative decrease}
\newacro{5G}{fifth generation}
\newacro{SISO}{single-input and single-output}
\newacro{MIMO}{multiple-input and multiple-output}
\newacro{MUI}{multi-user interference-plus-noise}
\newacro{MAC}{multiple access channel}
\newacro{PMAC}{parallel multiple access channel}
\newacro{CSI}{channel state information}
\newacro{CSIT}{channel state information at the transmitter}
\newacro{BS}{base station}
\newacro{TDD}{time-division duplexing}
\newacro{CDMA}{code division multiple access}
\newacro{FDMA}{frequency division multiple access}
\newacro{DSL}{digital subscriber line}
\newacro{SIC}{successive interference cancellation}
\newacro{SUD}{single user decoding}
\newacro{SINR}{signal-to-interference-and-noise ratio}
\newacro{WF}{water-filling}
\newacro{IWF}{iterative water-filling}
\newacro{SWF}{simultaneous water-filling}
\newacro{iid}[i.i.d.]{independent and identically distributed}
\newacro{OFDMA}{orthogonal frequency-division multiple access}
\newacro{DXL}{discounted exponential learning}
\newacro{AMXL}[MXL-a]{asynchronous matrix exponential learning}
\newacro{EXL}[MXL-e]{eigen-based exponential learning}
\newacro{FCC}{Federal Communications Commission}
\newacro{NTIA}{National Telecommunications and Information Administration}
\newacro{GAO}{General Accounting Office}
\newacro{QoE}{quality of experience}
\newacro{QoS}{quality of service}
\newacro{OFDM}{orthogonal frequency division multiplexing}\acused{OFDM}
\newacro{EW}{exponential weight}
\newacro{OGD}{online gradient descent}
\newacro{OMD}{online mirror descent}
\newacro{APT}{asymptotic pseudotrajectory}
\newacro{ICT}{internally chain transitive}
\newacro{EPA}{extended pedestrian A}
\newacro{EVA}{extended vehicular A}
\newacro{ETU}{extended typical urban}
\newacro{PPP}{Poisson point process}
\newacro{DL}{downlink}
\newacro{UL}{uplink}
\newacro{SU}{synchronous updates} 
\newacro{AU}{asynchronous updates}

\begin{abstract}
%
%
Motivated by applications to multi-antenna wireless networks, we propose a distributed and asynchronous algorithm for stochastic semidefinite programming.
This algorithm is a stochastic approximation of a continous-time matrix exponential scheme regularized by the addition of an entropy-like term to the problem's objective function.
We show that the resulting algorithm converges almost surely to an $\eps$-approximation of the optimal solution requiring only an unbiased estimate of the gradient of the problem's stochastic objective.
When applied to throughput maximization in wireless \ac{MIMO} systems, the proposed algorithm retains its convergence properties under a wide array of mobility impediments such as user update asynchronicities, random delays and/or ergodically changing channels.
Our theoretical analysis is complemented by extensive numerical simulations which illustrate the robustness and scalability of the proposed method in realistic network conditions.
\end{abstract}


\thanks{%
This research was supported by the European Comission in the framework of the QUANTICOL project (grant agreement no. 600708)
and the French National Research Agency under grant agreements
NETLEARN (ANR\textendash 13\textendash INFR\textendash 004)
and GAGA (ANR\textendash 13\textendash JS01\textendash 0004\textendash 01).}

\thanks{Part of this work was presented in ISIT 2014 \cite{CGM14}.}

\maketitle

\acresetall

\section{Introduction}
\label{sec:intro}

Semidefinite programming (i.e. the minimization of a convex function over a convex subset of the cone of positive-semidefinite matrices) comprises a rich class of convex optimization problems that is both relatively tractable (interior-point methods can often be used with polynomial worst-case complexity \cite{nemi94}) and also very powerful (many optimization problems in engineering and combinatorial optimization can be recast as semidefinite programs \cite{Boyd}).
Especially in an engineering context however, many applications involve a certain degree of randomness (either in the objective function itself or in the feedback provided to the optimizer) \cite{GV97,YSPP14} so many standard semidefinite optimization algorithms cannot be applied ``off the shelf''.
For instance, minimum volume convering problems (where quadratic functions  can be expressed as semidefinite constraints) have been a very active research topic for the last fifty years \cite{Sun}.
In wireless telecommunications, transmission ranges of mobile devices have also been modeled  as Euclidean balls with random parameters, hence expressible via semidefinite constraints with stochastic perturbations;
as a result, route discovery in mobile ad-hoc networks is typically addressed using stochastic semidefinite programming approaches \cite{Zhu}.
In view of the above, we focus in this paper on \emph{stochastic} semidefinite programming, a subclass of semidefinite programs where the objective function is given in the form of a stochastic expectation, with possibly unknown randomness.

In this framework, there are two main algorithmic approaches.
In the ``offline'' approach, it is assumed that the optimizing agent (or agents in the case of multi-agent optimization) knows the stochastic expectation of his objective function in some (semi-)explicit form (possibly quite complicated) and tries to optimize it by calling an appropriate semidefinite optimization algorithm.
On the other hand, in the ``online'' approach to optimization, the functional form of the objective function (and any inherent randomness) is unknown and the agent seeks to optimize his objective based on indirect (and possibly imperfect) performance indicators.
The former approach is usually employed in large-scale industrial optimization problems where the collection of data is not costly, but their processing is;
instead, the latter approach applies to distributed optimization problems in complex systems (such as networks) where the optimizing agents are not capable of collecting a lot of optimization data \textendash\ but the agents have the computing power to handle the data they collect.

Motivated by applications to wireless networks, our paper adopts the second approach with the aim of proposing a fully distributed algorithm for stochastic multi-agent semidefinite optimization problems that
\begin{inparaenum}
[\itshape a\upshape)]
requires minimal (and possibly imperfect) gradient information;
and
\item
it is fully parallelizable and does not require any coordination between the optimizing agents.
\end{inparaenum}
This algorithm is obtained as a variable step-size stochastic approximation \cite{Ben99,Bor08} of a continuous-time matrix exponential learning scheme which has important ties to the mirror descent machinery of \cite{NJLS09,Nes09,KSST12}.
In contrast to mirror descent methods however, we establish the convergence of the algorithm's \emph{last} iterate and not only the convergence of its empirical time-average, properly weighed by the step-size sequence employed.
In applications to wireless mobile systems, this is crucial because it implies the convergence of the network to a stable, optimum state in a strong sense instead of a weaker, average sense.

To complement our abstract theoretical analysis (Sections \ref{sec:problem} and \ref{sec:analysis}), we also present a concrete application to multi-antenna wireless mobile networks with ergodically changing channel conditions.
As explained in Section \ref{sec:applications}, this case fits squarely within the core stochastic semidefinite programming framework of Section \ref{sec:problem}:
First, this is due to the problem's inherently distributed aspect (since it is often impossible \textendash\ or imractical \textendash\ to coordinate and/or synchronize the mobile users' updates),
and, second, due to the lack of full system information at the user end and the fact that users do not necessarily know the stochastic law of their channels.

The users' objective in this setting is to maximize their information transmission rate by optimizing the covariance matrix of their input signal distribution.
Two cases are considered.
First, we consider the case where the users have perfect feedback from the receiver but their channels evolve following a stationary, ergodic process (the fast-fading regime) \cite{GV97};
in this case, the users' transmission rate is the stochastic average of their achievable rate over all channel realizations and the problem boils down to a multi-agent stochastic semidefinite program.
The second case concerns static channel conditions (i.e. the wireless medium is assumed to evolve at a much slower rate than the transmitters' update time-scale).
In this case, a major challenge arises if the users only have access to imperfect receiver feedback and \acl{CSI};
thus, even though the underlying problem is deterministic, stochasticity arises from the noise in the users' measurements and observations.
A partial description of our method applied to this ``imperfect information'' case was presented at our earlier conference paper \cite{CGM14}.

In both cases, we show how the proposed algorithmic scheme can be implemented in both synchronous and asynchronous ways.
Additionally, we also provide a procedure to compute an unbiased estimator of  the gradient of the transmission rate for each transmitter via receiver-transmitter reciprocity.
Finally, we also provide a suite of numerical simulations to illustrate the robustness of our algorithm and to compare it to more traditional water-filling techniques (which it outperforms).

\section{Problem Formulation and Preliminaries}
\label{sec:problem}

As we mentioned in the introduction, our main goal is to provide an efficient and robust solution method for semidefinite optimization problems where the objective function depends on a controlled matrix variable $\bX$ and a random variable $\omega$ (that cannot be controlled by the optimizer).
More precisely, we consider problems where, through repeated iterations, the optimizing agent (or \emph{agent} for short) seeks to converge to a value of $\bX$ that optimizes the expected value of the objective function with respect to $\omega$ (i.e. that solves the agent's stochastic optimization problem ``on average'').
Obviously, if the agent's ``mean'' objective function can be calculated explicitly, the above boils down to a deterministic problem;
however, a major challenge occurs if this expectation cannot be calculated \textendash\ or, worse, if the distribution of $\omega$ is not even known to begin with.

The above problem will comprise the core of our considerations and we will formalize it in the following section;
a variant formulation for multi-agent environments is then provided in Section \ref{sec:multi}.
From a mathematical point of view, both models are essentially equivalent but, from a practical standpoint, they describe problems of a very different nature.

\subsection{The core problem}
\label{sec:core}

Let $\herm_{M} = \{\bX\in\C^{M\times M}: \bX = \bX^{\dag}\}$ denote the space of $M\times M$ Hermitian matrices and let $\strat = \{\bX \mgeq0: \tr(\bX) = 1\}$ denote the spectrahedron of positive-semidefinite matrices with unit trace.
In what follows, we will focus on the \emph{stochastic semidefinite optimization problem:}
\begin{equation}
\label{eq:SSP}
\tag{SSP}
\begin{aligned}
\textrm{minimize}
	&\quad
	\ex[\obj(\bX;\omega)],
	\\
\textrm{subject to}
	&\quad
	\bX\in\strat,
\end{aligned}
\end{equation}
where $\omega$ is an abstract random variable taking values in some probability space $\Omega$,
the expectation $\ex[\argdot]$ is taken with respect to the law of $\omega$,
and $\obj\from\strat\times\Omega\to\R$ is a smooth random function which is convex with respect to $\bX\in\strat$ for all $\omega\in\Omega$.

Importantly, the simple formulation \eqref{eq:SSP} above accounts for a fairly wide class of stochastic optimization problems over compact spectrahedra (the semidefinite equivalent of polytopes, either real or complex).
In fact, as long as the feasible region $\strat'$ of a semidefinite program is a spectrahedron that is invariant under unitary transformations of the form $\bX\mapsto\bU\bX\bU^{\dag}$ for all unitary matrices $\bU$,%
\footnote{Recall here that a complex matrix $\bU$ is unitary if and only if $\bU \bU^{\dag} = \bU^{\dag} \bU = \bI$.
For the real case, invariance need only hold over all orthogonal matrices $\bO$ such that $\bO\bO^{\top} = \bO^{\top}\bO = \bI$.}
optimizing a convex function over $\strat'$ boils down to optimizing a convex function over $\strat$ (at the cost of increasing the problem's dimensionality) \cite{RG95}.
As such, \eqref{eq:SSP} can be seen as a canonical form for stochastic optimization problems over compact, unitary-invariant spectrahedra.

In the above framework, the \emph{mean objective function}
\begin{equation}
\label{eq:obj-mean}
\objmean(\bX)
	= \ex[\obj(\bX;\omega)]
\end{equation}
is itself convex over $\strat$;
for simplicity, we will also assume that $\objmean$ is finite and smooth over $\strat$.
In this way, we obtain the (convex) \emph{semidefinite optimization problem}
\begin{equation}
\label{eq:SP}
\tag{SP}
\begin{aligned}
\textrm{minimize}
	&\quad
	\objmean(\bX),
	\\
\textrm{subject to}
	&\quad
	\bX\in\strat,
\end{aligned}
\end{equation}
which could be solved by standard convex programming methods, provided that $\objmean$ is known to the optimizer.
As such, the main difficulty in solving \eqref{eq:SSP}/\eqref{eq:SP} is precisely that the law of $\omega$ may not be known, in which case the functional form of $\objmean$ is also unknown.
To circumvent this difficulty, standard results in convex analysis \cite{Str65} show that the gradient matrix
\begin{equation}
\label{eq:V}
\bV(\bX)
	= \grad_{\bX}\objmean(\bX)
\end{equation}
of the mean objective function $\objmean$ may be calculated by interchanging differentiation with expectation, i.e.
\begin{equation}
\label{eq:grad-mean}
\grad_{\bX}\objmean(\bX)
	= \ex[\grad_{\bX}\obj(\bX;\omega)]
	\quad
	\text{for all $\bX\in\strat$}.
\end{equation}
Thus, following the stochastic approximation approach of \cite{NJLS09,Nes09,Boyd}, we will focus on solving \eqref{eq:SSP} based only on random (sample-dependent) estimates of the stochastic gradient matrices $\grad_{\bX}\obj(\bX;\omega)$.%
\footnote{For posterity, we note here that $\grad_{\bX}\obj(\bX;\omega)$ is Hermitian (on account of the fact that $\obj$ is real).}

To make this precise, let $\bX(1),\bX(2),\dotsc$, be a (possibly random) sequence of play by the optimizing agent \textendash\ that is, at stage $n$, the agent chooses $\bX(n)$ and incurs an expected cost of $\objmean(\bX(n))$.
Then, at each stage $n=1,2,\dotsc$, we will assume that the agent has access to a random matrix $\hat\bV(n)$ which satisfies the statistical unbiasedness hypothesis:
\begin{assumption}
\label{asm:zeromean}
$\hat\bV(n)$ is a uniformly bounded random variable such that
\begin{equation}
\tag{A1}
\label{eq:zeromean}
\ex\big[\hat\bV(n) \given \filter_{n} \big]
	= \bV(\bX(n))
	\quad
	\text{for all $n=1,2,\dotsc$},
\end{equation}
where $\filter_{n}$ denotes the filtration induced by the history process $\bX(n)$.
\end{assumption}

%
%

The statistical hypothesis above allow us to account for a very wide range of estimation oracles:
in particular, we will \emph{not} be assuming \ac{iid} observations (a feature which is crucial in the context of wireless networks where observations are typically correlated with the state of the system).
Instead, we will only assume there is an oracle mechanism that returns a $\filter_{n}$-measurable estimate $\hat\bV(n)$ of $\bV(\bX(n))$ once the agent plays $\bX(n)$;
the construction of such an oracle for specific applications will be detailed in Section \ref{sec:applications}.

\begin{remark}
\label{rem:DetF}
An important special case of the problem \eqref{eq:SSP} is when the expectation in \eqref{eq:obj-mean} is \emph{deterministic}, i.e. $f(\argdot,\omega) = f(\argdot,\omega')$ for almost every $\omega,\omega'\in\Omega$.
In that case, Assumption \eqref{eq:zeromean} accounts for problems where the optimizer is called to solve a deterministic semidefinite program with imperfect gradient feedback and stochasticity stems from the random noise perturbing the agent's observations.
More generally, depending on the structure of the probability space $\Omega$, the randomness in the stochastic optimization problem \eqref{eq:SSP} and the randomness in the gradient observations $\hat\bV(n)$ could be completely decoupled;
the only assumption that we will make regarding these different degrees of randomness is \eqref{eq:zeromean}.
\end{remark}

\subsection{Multi-agent optimization and games}
\label{sec:multi}

In multi-agent en\-vi\-ron\-ments, we assume that there are multiple optimizing agents $k=1,\dotsc,K$, each one controlling an individual control variable $\bX_{k}$ that impacts the agents' global objective $\obj$ in a different way.
Specifically, this amounts to the following multi-agent version of \eqref{eq:SSP}:
\begin{equation}
\label{eq:SSP-multi}
\begin{aligned}
\textrm{minimize}
	&\quad
	\ex\big[\obj(\bX_{1},\dotsc,\bX_{K};\omega) \big],
	\\
\textrm{subject to}
	&\quad
	\bX_{k}\in\strat_{k},
\end{aligned}
\end{equation}
where $\strat_{k} = \{ \bX_{k}\in\C^{M_{k}\times M_{k}}: \bX_{k}\mgeq0,\,\tr(\bX_{k}) = 1 \}$ denotes the feasible region of agent $k$ and $f\from\prod_{k}\strat_{k}\times\Omega\to\R$ satisfies the same convexity and smoothness assumptions as before.

In this setting, if there is no central controller to coordinate the agents' actions and provide global feedback, it will be assumed that agents can only access an estimate $\hat\bV_{k}$ of their \emph{individual gradient matrices}
\begin{equation}
\label{eq:Vk}
\bV_{k}(\bX)
	= \grad_{\bX_{k}} \objmean(X)
\end{equation}
where $\bX = (\bX_{1},\dotsc,\bX_{K})$ denotes the agents' aggregate action profile and $\objmean(\bX) = \ex[\obj(\bX;\omega)]$.
Thus, mutatis mutandis, we will assume that Assumption \eqref{eq:zeromean} applies to each agent separately, and we will seek to provide a \emph{distributed optimization} algorithm that solves \eqref{eq:SSP-multi} under these assumptions.

As a further extension of the above framework, we will also consider the case where each agent seeks to minimize unilaterally an individual objective function $\obj_{k}$ (i.e. there is no global objective).
This situation is known as a \emph{game in normal form} (or, more simply, a \emph{game}) and the solution concept that we will focus on is that of Nash equilibrium \cite{Nas51,Deb52,Ros65,MS96}.
Formally, we will say that an action profile $\eq = (\eq_{1},\dotsc,\eq_{K})$ is a \emph{Nash equilibrium} of the game induced by the mean individual objective functions $\objmean_{k}(\bX) = \ex[\obj_{k}(\bX;\omega)]$ when
\begin{equation}
\label{eq:Nash}
\tag{NE}
\objmean_{k}(\eq)
	\leq \objmean_{k}(\bX_{k};\eq_{-k})
\end{equation}
for every unilateral deviation $\bX_{k}\in\strat_{k}$ and for every agent $k=1,\dotsc,K$,
with $(\bX_{k};\eq_{-k})$ denoting the tuple $(\eq_{1},\dotsc,\bX_{k},\dotsc,\eq_{K})$.
Put differently, Nash equilibria are simply action profiles which are \emph{unilaterally stable} in that no agent has any incentive to deviate from them.

The connection between game theory and distributed optimization is recovered in the class of \emph{potential games} \cite{MS96}, i.e. games where the players' mean objective functions are aligned along a common potential function $\objmean$.
More precisely, following Monderer and Shapley \cite{MS96}, we will say that $\objmean$ is a \emph{potential function} for a game with mean objectives $\objmean_{k}$ when
\begin{equation}
\label{eq:potential}
\objmean_{k}(\bX_{k};\bX_{-k}) - \objmean_{k}(\bX_{k}';\bX_{-k})
	= \objmean(\bX_{k};\bX_{-k}) - \objmean_{k}(\bX_{k}';\bX_{-k})
\end{equation}
for all actions $\bX_{k}, \bX_{k}'\in\strat_{k}$ of agent $k$, and for all action profiles $\bX_{-k}\in\strat_{-k} \equiv \prod_{\ell\neq k} \strat_{\ell}$ of $k$'s opponents.
As can be easily seen, if a game admits a potential function, its Nash equilibria necessarily coincide with the critical points of its potential function \cite{MS96}.
Thus, if the game's potential $\objmean$ is convex over $\strat \equiv \prod_{k}\strat_{k}$, it follows that the equilibrium problem \eqref{eq:Nash} can be reduced to the distributed optimization problem \eqref{eq:SSP-multi}.%
\footnote{Conversely, every distributed optimization problem can be seen as a potential game by setting $\obj_{k} = \obj$ for all $k$.}
We will use this observation freely throughout our paper \textendash\ and, especially, in Section \ref{sec:applications}.

\section{Algorithms and Results}
\label{sec:analysis}

\subsection{Single-agent optimization analysis}

The main algorithmic scheme that we will use to solve \eqref{eq:SSP}/\eqref{eq:SP}
will be based on the following \ac{DXL} recursion:
\begin{equation}
\label{eq:DXL}
\tag{DXL}
\begin{aligned}
\bY(n+1)
	&= \bY(n) - \step_{n} \left( \hat\bV(n) + \temp\bY(n) \right),
	\\
\bX(n+1)
	&= \frac{\exp(\bY(n+1))}{\tr[ \exp(\bY(n+1)) ]},
\end{aligned}
\end{equation}
where:
\begin{enumerate}
\item
$\bY(n)$ is an auxiliary scoring matrix which aggregates gradient information.
\item
$\hat\bV(n)$ is a random matrix variable satisfying the unbiasedness assumption \eqref{eq:zeromean}.
\item
$\step_{n}$, $n=1,2,\dotsc$, is a nonincreasing step-size sequence (specific assumptions for $\step_{n}$ will be discussed below).
\item
$\temp>0$ is a (small) discount parameter which acts as a failsafe against the iterates $\bY(n)$ getting out of bounds.
\end{enumerate}

Intuitively, \eqref{eq:DXL} acts as a ``regularized'' stochastic gradient descent process:
if we ignore the parameter $\temp$ for the moment, each iteration of \eqref{eq:DXL} simply aggregates the received gradient information (in the update of $\bY$) and then ``projects'' back to the primal variable $\bX$ to receive a new gradient and continue the process.
The reason that the exponentiation step acts as a ``projection''
operator is that it aligns the eigenvalues of $\bX$ with those of
$\bY$, so, in a certain sense, $\bX$ is an ``exponential projection''
of $\bY$ to $\strat$.

\begin{algorithm}[t]
{%
\flushleft
\sf
\vspace{2pt}
Parameters:
discount parameter $\temp>0$;
decreasing step-size sequence $\step_{n}$.
\\
Initialize:\;
$n \leftarrow 0$;\;
$\bY \leftarrow 0$;\;
$\bX \leftarrow \exp(\bY)/\tr[\exp(\bY)]$;
\\[2pt]
\Repeat
{%
\ForEach
{agent $k \in \play$}
	{\vspace{1ex}
	get gradient estimate $\hat\bV$;
	\\[2pt]
	update score matrix:
	$\bY \leftarrow \bY - \step_{n} \left( \hat\bV + \temp \bY \right)$;
	\\[2pt]
	update primal variable:
	$\dis \bX \leftarrow \exp(\bY) \big/\tr[\exp(\bY)]$;
	\\[2pt]
	$n \leftarrow n+1$;
} 
\vspace{.5ex}
\textbf{until} termination criterion is reached.} 
} 
\caption{Algorithmic implementation of \eqref{eq:DXL}.}
\label{alg:DXL}
\end{algorithm}

The role of the discount parameter $\temp$ in \eqref{eq:DXL} (and the reason for calling it a ``discount'' in the first place) is more subtle.
To understand it, note first that the recursive step of \eqref{eq:DXL} can be rewritten in aggregate form as: 
\begin{equation}
\label{eq:discount}
\bY(n+1)
	= e^{-T_{n,1}} \bY(1)
	- \sum_{m=1}^{n} e^{-T_{n,m+1}} \step_{m} \hat\bV(m),
\end{equation}
where, assuming that $\step_{n}$ is small enough, we have set:
\begin{equation}
\label{eq:time-discounted}
T_{n,m}
	= \sum_{j=m}^{n} \log(1 - \temp\step_{j}).
\end{equation}
By expanding the logarithm to leading order in \eqref{eq:time-discounted}, this last sum is asymptotically equal to $-\temp t_{n,m}$ where
\begin{equation}
\label{eq:time-step}
t_{n,m}
	= \sum_{j=m}^{n} \step_{j}.
\end{equation}
Accordingly, to leading order, \eqref{eq:discount} can be rewritten for large $n$ (and small $\step_{n}$) as:
\begin{equation}
\label{eq:discount-simple}
\bY(n+1)
	\approx \discrate^{t_{n,1}} \bY(1)
	- \sum_{m=1}^{n} \discrate^{t_{n,m+1}} \step_{m} \hat\bV(m),
\end{equation}
with $t_{n,m}$ given by \eqref{eq:time-step} and $\discrate = \exp(-\temp)$.

Of course, the above derivation is approximate in nature but it highlights the discount role of $\temp$.
For a constant step-size sequence $\step_{n} = \step$, we have $t_{n,m} = \step \cdot (n-m)$, so the exponential sum in \eqref{eq:discount-simple} means that \eqref{eq:DXL} assigns (exponentially) more weight to recent observations rather than older ones.
In a sense, this discounting counters the use of a vanishing $\step_{n}$.
A decreasing step-size implies that more recent gradient observations enter the algorithm with a decreasing weight;
by contrast, the use of a positive discount parameter $\temp>0$ tempers this (somewhat counter-intuitive) behavior by increasing the relative weight of more recent gradient observations.
Moreover, from a calculational standpoint, the use of a positive discount parameter $\temp$ has the added benefit that the auxiliary score matrices $\bY(n)$ cannot grow too large.
If the step-size sequence $\step_{n}$ is chosen in a way such that the geometric series $\insum_{m=1}^{n} \temp_{m} \discrate^{t_{n,m+1}}$ remains summable,%
\footnote{We will elaborate more on the choice of $\step_{n}$ below.}
then $\bY(n)$ will be uniformly bounded on account of \eqref{eq:discount-simple} and Assumption \eqref{eq:zeromean}.
Since computing d

Of course, in so doing, the discount parameter $\temp$ also introduces a systematic deterministic bias to the gradient observations $\hat\bV(n)$, i.e. a perturbation that persists even in the noiseless regime where $\hat\bV(n)$ is actually deterministic.
Indeed, if \eqref{eq:DXL} is run with perfect gradient observations $\hat\bV(n) = \bV(\bX(n))$, then any fixed point $\eq$ of \eqref{eq:DXL} will satisfy:
\begin{equation}
\label{eq:fixed}
\begin{aligned}
\temp\bY^{\ast}
	&= \bV(\eq),
	\\
\eq
	&= \frac{\exp(\bY^{\ast})}{\tr[\exp(\bY^{\ast})]}.
\end{aligned}
\end{equation}
Setting $\bV^{\ast} = \bV(\eq)$ for convenience and solving \eqref{eq:fixed} for $\eq$ then gives:
\begin{equation}
\bV^{\ast} + \temp \log\eq
	= - \temp \log \tr\big[ \exp(-\bV^{\ast}/\temp) \big] \bI,
\end{equation}
or, after a slight rearrangement:
\begin{equation}
\label{eq:fixed-2}
\bV^{\ast} + \temp \left( \log \eq + \bI \right)
	= -\kappa \bI,
\end{equation}
for $\kappa = \temp \left( 1 + \tr[\exp(-\bV^{\ast}/\temp)] \right)$.
Importantly, the RHS of \eqref{eq:fixed-2} can be written more simply as $\bV^{\ast} + \temp \left( \log \eq + \bI \right) = \grad \objmean_{\temp}(\eq)$ where the perturbed objective function $\objmean_{\temp}\from\strat\to\R$ is defined as:
\begin{equation}
\label{eq:obj-temp}
\objmean_{\temp}(\bX)
	= F(\bX) + h(\bX),
\end{equation}
with
\begin{equation}
\label{eq:entropy}
h(\bX)
	= \tr[ \bX \log\bX ]
\end{equation}
denoting the so-called \emph{von Neumann} (or \emph{quantum}) \emph{entropy} of $\bX$ \cite{Wil67}.%
\footnote{That the gradient of $\objmean_{\temp}$ is $\grad \objmean_{\temp}(\bX) = \bV(\bX) + \temp(\log \bX + \bI)$ follows from standard arguments in matrix calculus \textendash\ see e.g. \cite[Appendix D]{Dat05}.}
Thus, given that $\eq$ must satisfy the trace constraint $\tr(\eq) = 1$, it follows that any fixed point $\eq$ of \eqref{eq:DXL} will be a solution of the perturbed optimization problem:
\begin{equation}
\label{eq:SP-temp}
\tag{SP$_{\temp}$}
\begin{aligned}
\textrm{minimize}
	&\quad
	\objmean(\bX) + \temp \tr[\bX \log \bX],
	\\
\textrm{subject to}
	&\quad
	\bX\in\strat.
\end{aligned}
\end{equation}

Obviously, the solution set of \eqref{eq:SP-temp} is asymptotically close to that of the unperturbed problem \eqref{eq:SP} in the limit $\temp\to0$ (where the entropic perturbation term $h(\bX)$ vanishes):
more precisely, if $\eq_{\temp}$ is a solution of \eqref{eq:SP-temp}, there exists a solution $\eq$ of \eqref{eq:SP} such that the distance between $\eq_{\temp}$ and $\eq$ vanishes as $\temp\to0$.
That said, an important difference between \eqref{eq:SP} and \eqref{eq:SP-temp} is that the latter is strictly convex (because $h$ is).
As a result, \eqref{eq:SP-temp} admits a unique solution, even when the solution set of \eqref{eq:SP} is a non-singleton convex set.

With all this in mind, we are in a position to state our main result for \eqref{eq:DXL}:

\begin{theorem}
\label{thm:conv-main}
Assume that \eqref{eq:DXL} is run with gradient observations satisfying \eqref{eq:zeromean} and with a variable step-size sequence $\step_{n}$ such that $\sum_{n=1}^{\infty} \step_{n}^{2} < \sum_{n=1}^{\infty} \step_{n} = \infty$.
Then, the iterates $\bX(n)$ of \eqref{eq:DXL} converge almost surely to a solution of the perturbed optimization problem \eqref{eq:SP-temp};
in particular, $\bX(n)$ converges \textup(a.s.\textup) within $\eps(\temp)$ of a solution of the stochastic problem \eqref{eq:SSP} and the error $\eps(\temp)$ vanishes in the limit $\temp\to0$.
\end{theorem}

Theorem \ref{thm:conv-main} will be our main result for \eqref{eq:DXL} so, before proving it, some remarks are in order:

\begin{remark}
The statement of Theorem \ref{thm:conv-main} suggests that the discount parameter $\temp$ should be taken as small possible in order to ensure the algorithm's convergence to a state $\eq_{\temp}\in\strat$ that is as close as possible to the solution set of \eqref{eq:SSP}.
On the other hand, very small $\temp>0$ could mean that the iterates $\bY(n)$ of \eqref{eq:DXL} could grow quite large, potentially exceeding the numerical capacity of the optimizing's agent calculating device \textendash\ recall the discussion surrounding \eqref{eq:discount-simple}.
As a result, the discount parameter $\temp>0$ essentially reflects the algorithm's accuracy vs. memory trade-off:
lower values of $\temp > 0$ lead to better solutions of \eqref{eq:SSP}, but at the expense of higher memory requirements and more processing power.
Ultimately, the choice of $\temp$ relies on the technical specifications of the optimization problem to be solved so the ``optimal'' choice of $\temp$ can only be made on a case-by-case basis.
\end{remark}

\begin{remark}
In a similar vein to the above remark, Assumption \ref{asm:zeromean} can actually be relaxed to account for gradient observations that are only bounded in mean squre (instead of being bounded almost surely).
In this context however, a given observation $\hat\bV$ of $\bV$ could exceed the storage/processing capacity of the agent's optimizing device, thus introducing additional arithmetic stability errors to running \eqref{eq:DXL}.
Such issues lie beyond the scope of the current work so we opted to work with the almost sure boundedness assumption for simplicity.
\end{remark}

\begin{remark}
We should also note here that the ``$\ell^{2} - \ell^{1}$'' summability condition $\sum_{n=1}^{\infty} \step_{n}^{2} < \sum_{n=1}^{\infty} \step_{n} = \infty$ can also be relaxed in the context of Assumption \ref{asm:zeromean}.
Specifically, Theorem \ref{thm:conv-main} remains true even with significantly more aggressive step-size sequences of the form $\step_{n} = n^{-a}$ for some arbitrarily small $a>0$.
The reason for stating (and proving) Theorem \ref{thm:conv-main} in the ``$\ell^{2} - \ell^{1}$'' framework was only done for simplicity;
in practice, the use of  a (nearly) constant step-size greatly accelerates the algorithm, a fact that we explore in Section \ref{sec:applications}.
\end{remark}

Now, to prove Theorem \ref{thm:conv-main}, our strategy will be as follows:
First, we will show that the iterates of \eqref{eq:DXL} constitute a so-called \ac{APT} of the mean, continuous-time dynamics:
\begin{equation}
\label{eq:DXL-cont}
\tag{DXL$_{c}$}
\begin{aligned}
\dot\bY
	&= -\bV(\bX) - \temp \bY,
	\\
\bX
	&= \frac{\exp(\bY)}{\tr[ \exp(\bY) ]},
\end{aligned}
\end{equation}
i.e. the iterates of \eqref{eq:DXL} are asymptotically close to solution segments of \eqref{eq:DXL-cont} of arbitrary length \cite{Ben99}.
We will then show that \eqref{eq:DXL-cont} converges to the (unique) solution of the perturbed optimization problem \eqref{eq:SP-temp};
the claim of \eqref{thm:conv-main} will then follow from standard results in the theory of stochastic approximation \cite{Ben99}.

We begin by showing that the iterates of \eqref{eq:DXL} comprise an \acl{APT} of the dynamics \eqref{eq:DXL-cont} in the sense of \cite{Ben99}, i.e.
\begin{equation}
\label{eq:APT}
\lim_{t\to\infty} \sup_{0\leq h \leq T} \norm{\bar\bX(t+h) - \Phi_{h}(\bar\bX(t))}
	= 0
	\quad
	\text{(a.s.)},
\end{equation}
where $\bar\bX(t)$, $t\geq0$ is the linear interpolation of the iterates $\bX(n)$ of \eqref{eq:DXL} while $\Phi_{t}(\bX)$ denotes the \emph{flow} induced on $\strat$ by \eqref{eq:DXL-cont} \textendash\ i.e. $\Phi_{t}(\bX)$, $t\geq0$, is the solution trajectory of \eqref{eq:DXL-cont} that starts at $\bX\in\strat$.
To that end, we will first need the following boundedness result:

\begin{lemma}
\label{lem:bounded}
If $\step_{n} < 1/\temp$ for all sufficiently large $n$, then the iterates $\bY(n)$ of \eqref{eq:DXL} under Assumption \ref{asm:zeromean} are bounded \textup(a.s.\textup).
\end{lemma}

\begin{proof}
First, let $\vbound>0$ be such that $\smallnorm{\hat\bV(n)} \leq \vbound$ almost surely (that such a $\vbound$ exists is a consequence of Assumption \ref{asm:zeromean});
additionally, let $n_{0}$ be such that $0 < 1 - \step_{n} \temp \leq 1$ for all $n \geq n_{0}$.
Then, for $n \geq n_0$, the definition \eqref{eq:DXL} of $\bY(n)$ and the bound $\smallnorm{\hat\bV(n)} < \vbound$ readily yield $\smallnorm{\bY(n+1)} \leq (1 - \temp \step_{n}) \smallnorm{\bY(n)} + \step_{n} \vbound$.
We are thus reduced to the following cases:
\begin{itemize}
\item
If $\temp \smallnorm{\bY(n)} \geq \vbound$, then $\smallnorm{\bY(n+1)} \leq \smallnorm{\bY(n)}  + \step_{n} (\vbound - \temp \smallnorm{\bY(n)} ) \leq \smallnorm{\bY(n)}$, so $\bY(n)$ decreases in norm.
\item
Otherwise, if $\temp \smallnorm{\bY(n)} < \vbound$, we will have $\smallnorm{\bY(n+1)} \leq (1 - \step_{n} \temp) \vbound/\temp  + \step_{n} \vbound = \vbound/\temp$.
\end{itemize}
It follows that $\smallnorm{\bY(n+1)}$ will either decrease or be uniformly bounded by $\vbound$, so our claim follows by induction.
\end{proof}

Thanks to this lemma, we readily obtain:
\begin{proposition}
\label{prop:APT}
With notation as in Lemma \ref{lem:bounded}, the sequence $\bY(n)$ comprises an \acl{APT} of \eqref{eq:DXL-cont}.
\end{proposition}

\begin{proof}
First, taking expectations in the RHS of \eqref{eq:DXL} yields:
\begin{equation}
\label{eq:mean}
\ex[\bY(n) - \step_{n} (\hat\bV(n) + \temp \bY(n)) \given \filter_{n}]
	= \bY(n) - \step_{n}( \bV(\bX(n)) + \temp \bY(n) ),
\end{equation}
where we used Assumption \ref{asm:zeromean} and the fact that $\bY(n)$ and $\bX(n)$ are fully determined by $\filter_{n}$.
Since $\bY(n)$ is bounded (a.s.) by Lemma \ref{lem:bounded}, our claim follows from Proposition 4.1 in \cite{Ben99}.
\end{proof}

We now proceed to show that the dynamics \eqref{eq:DXL-cont} converge to the (unique) solution of the perturbed optimization problem \eqref{eq:SP-temp} from any initial condition $\bY(0)$.
To that end, we will first need to derive the dynamics of the primal control variable $\bX(t)$:

\begin{lemma}
\label{lem:DXL-primal}
Let $\bX(t)$ be a solution orbit of the continuous-time dynamical system \eqref{eq:DXL-cont}.
Then, $\bX(t)$ satisfies the dynamics:
\begin{equation}
\label{eq:DXL-primal}
\dot\bX
	= - \int_{0}^{1} \bX^{1-s} \bV_{\temp}(\bX) \bX^{s} \dd s
	+ \tr[\bX \bV_{\temp}(\bX)] \bX,
\end{equation}
where
\begin{equation}
\label{eq:V-temp}
\bV_{\temp}(\bX)
	= \bV(\bX) + \temp \log \bX.
\end{equation}
\end{lemma}

\begin{proof}
Let $\pf(\bY) = \tr[\exp(\bY)]$.
Then, differentiating $\bX(t)$ with respect to $t$, we get:
\begin{flalign}
\dot\bX
	&= \frac{1}{\pf(\bY)} \frac{d}{dt} \exp(\bY) - \frac{\exp(\bY)}{\pf^{2}(\bY)} \dot\pf
	\notag\\
	&= \frac{1}{\pf(\bY)} \int_{0}^{1} e^{(1-s)\bY} \dot \bY e^{s\bY} \dd s
	- \frac{1}{\pf^{2}(\bY)} e^{\bY} \tr[\dot\bY e^{\bY}]
	\notag\\
	&= -\int_{0}^{1} \bX^{1-s} \bV_{\temp}(\bX) \bX^{s} \dd s
	+ \tr[\bX \bV_{\temp}(\bX)] \bX,
\end{flalign}
where the second equality is an application of Fréchet's derivative formula for matrix exponentials \cite{Hig08} and the last one follows by recalling that $\bX = \exp(\bY) / \tr[\exp(\bY)]$ so $\dot \bY = -\bV(\bX) - \temp ( \log\bX + \pf(\bY)\bI ) = -\bV_{\temp}(\bX) - \temp\pf(\bY) \bI$ by the definition of the dynamics \eqref{eq:DXL-cont}.
\end{proof}

With this explicit expression for the evolution of $\bX$ at hand, we are almost in a position to show that the perturbed objective function $\objmean_{\temp}(\bX) = \objmean(\bX) + \temp \tr[\bX \log\bX]$ of \eqref{eq:SP-temp} is a strict Lyapunov function for the dynamics \eqref{eq:DXL-primal}.
The only other result that we will need is the following Jensen-like inequality for positive-definite matrices:

\begin{lemma}
\label{lem:Jensen}
Consider Hermitian matrices $\bW, \bX\in\herm_{M}$ with $\bX\mg0$ and $\tr(\bX) = 1$.
Then, for all $s\in[0,1]$, we have $\tr(\bX^{1-s} \bW \bX^{s} \bW) \geq \tr(\bX\bW)^{2}$ with equality if and only if $\bW\propto \bI$.
\end{lemma}

\begin{proof}
Let $a=(1-s)/2$, $b=s/2$, and set $\bA = \bX^{1/2}$, $\bB = \bX^{a} \bW \bX^{b}$.
Then, the Cauchy\textendash Schwarz inequality for matrices gives $\tr(\bA \bA^{\dag}) \tr(\bB \bB^{\dag}) \geq \smallabs{\tr(\bA \bB^{\dag})}^{2}$ with equality iff $\bA\propto \bB$.
On the other hand, we also have $\tr(\bA\bA^{\dag}) = \tr\bX = 1$ and $\tr(\bB\bB^{\dag}) = \tr[\bX^{a} \bW \bX^{b} \bX^{b} \bW \bX^{a}] = \tr[\bX^{1-s} \bW \bX^{s} \bW]$, leading to the inequality:
\begin{equation}
1\cdot \tr[\bX^{1-s} \bW \bX^{s} \bW]
	\geq \abs{\tr[\bX^{1/2} \bX^{s/2} \bW \bX^{(1-s)/2}]}^{2}
	= \smallabs{\tr (\bX\bW)}^{2}
	= \tr(\bX \bW)^{2},
\end{equation}
where the last equality follows from the fact that $\tr(\bX \bW)$ is real (recall that $\bX$ is positive-definite while $\bW$ is Hermitian).
This inequality holds as an equality if and only if $\bX^{1/2} \propto \bX^{a} \bW \bX^{b}$ so, with $a+b = 1/2$, this last condition is equivalent to $\bW \propto \bI$, as claimed.
\end{proof}

With all this in hand, we obtain:

\begin{proposition}
\label{prop:Lyapunov}
Let $\bX(t)$ be an interior solution orbit of the continuous-time dynamics \eqref{eq:DXL-cont}.
Then, $\frac{d}{dt} \objmean_{\temp}(\bX(t)) \leq 0$ for all $t\geq0$, with inequality if and only if $\bV_{\temp}(\bX(t)) \propto \bI$ \textendash\ i.e. at interior stationary points of \eqref{eq:DXL-primal}.
\end{proposition}

\begin{proof}
By a simple application of the chain rule, we readily get:
\begin{equation}
\dot\objmean_{\temp}
	= \tr[ \dot\bX \grad\objmean_{\temp}(\bX) ]
	= -\tr[ \dot\bX \cdot (\bV(\bX) + \temp(\log\bX + \bI)) ]
	= -\tr[ \dot\bX \bV_{\temp}(\bX)],
\end{equation}
where we have used the definition of $\bV_{\temp}$ and the fact that $\tr[\dot\bX] = 0$ (since $\tr[\bX(t)] = 1$ for all $t\geq0$).
Invoking Lemma \ref{lem:DXL-primal}, we then obtain
\begin{flalign}
\dot \objmean_{\temp}
	&= \int_{0}^{1} \tr\big[\bX^{1-s} \bV_{\temp}(\bX) \bX^{s} \bV_{\temp}(\bX) \big] \dd s
	- \tr[ \bX \bV_{\temp}(\bX) ]^{2}
	\notag\\
	&= \int_{0}^{1} \tr\big[\bX^{1-s} \bV_{\temp}(\bX) \bX^{s} \bV_{\temp}(\bX) \big]
	- \tr[ \bX \bV_{\temp}(\bX) ]^{2} \dd s,
\end{flalign}
and our assertion follows from Lemma \ref{lem:Jensen} above.
\end{proof}

As a corollary of the above, we then get:

\begin{corollary}
\label{cor:conv-cont}
For every initial condition $\bY(0)\in\herm_{M}$, the dynamics \eqref{eq:DXL-cont} converge to the unique solution $\eq_{\temp}$ of the perturbed optimization problem \eqref{eq:SP-temp}
\end{corollary}

Finally, we have:

\begin{proof}[Proof of Theorem \ref{thm:conv-main}]
By Proposition \ref{prop:APT}, the iterates of \eqref{eq:DXL} form an \acl{APT} of the continuous-time dynamical system \eqref{eq:DXL-cont}.
Since the objective function of the perturbed optimization problem \eqref{eq:SP-temp} is a strict Lyapunov function for the latter (Proposition \ref{prop:Lyapunov} coupled with the fact that any solution of \eqref{eq:SP-temp} is interior), our claim follows readily from standard stochastic approximation results \cite[Theorem 5.7]{Ben99}.
\end{proof}

From the proof of Theorem \ref{thm:conv-main}, we can identify two points where the positivity of $\temp$ plays a crucial role.
The first is the boundedness of the iterates $\bY(n)$ of the algorithm (Lemma \ref{lem:bounded}) which guarantees that \eqref{eq:DXL} is a stochastic approximation of the mean dynamics \eqref{eq:DXL-cont}.
The second is the fact that the problem \eqref{eq:SP-temp} admits a unique, interior solution.
In the limit case $\temp=0$, it is still possible to show that \eqref{eq:DXL} comprises an \acl{APT} of \eqref{eq:DXL-cont} but the Lyapunov argument of Proposition \ref{prop:Lyapunov} is more subtle.
Since we are only interested in algorithms with finite iterates (for computer arithmetic reasons), we will not press this issue further, delegating it instead to future work.

\subsection{Distributed optimization in asynchronous multi-agent
  environments}
\label{ssec:async}

Of course, even though the information requirements of \eqref{eq:DXL} are relatively minimal (an imperfect oracle call to the gradient of the agent's stochastic objective), it is not clear whether it can be readily extended to a distributed optimization setting (or a game-theoretic context) where agents update independently of one another and there is often a delay between agent updates and observations.
To overcome these limitations, we examine here a fully decentralized variant of \eqref{eq:DXL} which addresses the issues above.

To make all this precise, we will work with the multi-agent stochastic optimization problem \eqref{eq:SSP-multi} and we will assume that the agents seek to converge to a solution thereof through repeated play.
To that end, let $n$ denote the $n$-th \emph{overall} update epoch in the system, let $\play_{n} \subset \play$ denote the subset of agents who update at this epoch (typically $\abs{\play_{n}} = 1$ if agents update at random times), and let $d_{k}(n)$ be the number of periods that have elapsed at period $n$ since the last update of agent $k$.
With all this in mind, we will focus on the following asynchronous variant of \eqref{eq:DXL}:
\begin{equation}
\label{eq:DXL-multi}
\begin{aligned}
\bY_{k}(n+1)
	&= \bY_{k}(n) - \step_{n_{k}} \one(k\in\play_{n}) \cdot \big( \hat\bV_{k}(n) + \temp \bY_{k}(n) \big),
	\\
\bX_{k}(n+1)
	&= \frac{\exp(\bY_{k}(n+1))}{\tr[ \exp(\bY_{k}(n+1)) ]},
\end{aligned}
\end{equation}
where $n_{k} = \sum_{j=1}^{n} \one(k\in\play_{j})$ denotes the number of updates performed by agent $k$ up to epoch $n$ while the (asynchronous) gradient estimate $\hat\bV_{k}(n)$ satisfies the unbiasedness assumption:
\begin{equation}
\label{eq:grad-multi}
\tag{\ref*{eq:zeromean}$'$}
\ex\big[ \hat\bV_{k}(n) \given \filter_{n} \big]
	= \bV_{k}(\bX_{1}(n - d_{1}(n)),\dotsc,\bX_{K}(n - d_{K}(n))),
\end{equation}
where, as before, $\bV_{k}(\bX_{1},\dotsc,\bX_{n}) = \grad_{\bX_{k}} \objmean(\bX_{1},\dotsc,\bX_{K})$.

\begin{algorithm}[t]
{%
\flushleft
\sf
\vspace{2pt}
Parameters:
	discount rate $\temp>0$;
	initial step-size $\step$.
	\\[2pt]
Initialize:
	$n \leftarrow 1$;
	$\bY \leftarrow 0$;
	$\bX \leftarrow \exp(\bY)/\tr[\exp(\bY)]$.
	\\[2pt]
\Repeat
{%
\texttt{UpdateEvent} occurs at time $\tau(n)$;
	\\[2pt]
	get gradient estimate $\hat\bV$;
	\\[2pt]
	update score matrix:
	$\bY \leftarrow \bY + \step/n \, \hat \bV$;
	\\[2pt]
	update primal variable:
	$\bX \leftarrow \exp(\bY) \big/ \tr[ \exp(\bY) ]$;
	\\[2pt]
	$n \leftarrow n+1$;
	\\[2pt]
\textbf{until} termination criterion is reached.
} 
} 
\caption{Asynchronous implementation of \eqref{eq:DXL}.}
\label{alg:DXL-multi}
\end{algorithm}

By definition, $\bY_{k}(n)$ and $\bX_{k}(n)$ are updated at the $(n+1)$-th update period if and only if $k\in\play_{n}$, so every agent follows his individual update timer, independently of what other agents in the system do (for a pseudocode implementation, see Algorithm \ref{alg:DXL-multi}).
Remarkably, in this completely decentralized context (with out-of-date and/or imperfect gradient observations), we still get:

\begin{theorem}
\label{thm:conv-multi}
Assume that the agents' delay process $d_{k}(n)$ are bounded \textup(a.s.\textup) and the set of agents $\play_{n}$ that updates at the $n$-th overall update epoch is a homogeneous recurrent Markov chain \textendash\ i.e. all agents update a strictly positive rate.
Assume further that Algorithm \ref{alg:DXL-multi} is run with step-sizes $\step_{n} \propto 1/n$ and imperfect gradient estimates $\hat\bV_{k}(n)$ satisfying the unbiasedness assumption \eqref{eq:grad-multi}.
Then, the algorithm's iterates converge \textup(a.s.\textup) to the \textup(unique\textup) minimizer of the perturbed objective $\objmean_{\temp}(\bX_{1},\dotsc,\bX_{K}) = \objmean(\bX_{1},\dotsc,\bX_{K}) + \temp \sum_{k=1}^{K} \tr[ \bX_{k} \log\bX_{k} ]$ over $\strat = \prod_{k=1}^{K} \strat_{k}$.

In particular, Algorithm \ref{alg:DXL-multi} converges within $\eps(\temp)$ of a solution of the distributed stochastic optimization problem \eqref{eq:SSP-multi} and the approximation error $\eps(\temp)$ vanishes as $\temp\to0^{+}$.
\end{theorem}

\begin{proof}
Following Theorems 2 and 3 in \cite{Bor08}, the asynchronous recursion \eqref{eq:DXL-multi} may be seen as a stochastic approximation of the rate-adjusted dynamics:
\begin{equation}
\dot\bY_{k}
	= -\rho_{k} [\bV_{k} + \temp\bY_{k}],
\end{equation}
where $\rho_{k} = \lim_{n\to\infty} n_{k}/n > 0$ is the asymptotic update rate of user $k$ (the existence and positivity of this limit follows from the ergodicity assumption on the set-valued process $\play_{n}$).
This multiplicative factor does not alter the rest points of the original dynamics \eqref{eq:DXL-cont} and an easy calculation shows that the perturbed objective function $\objmean_{\temp}(\bX_{1},\dotsc,\bX_{K})$ remains a strict Lyapunov function for the rate-adjustment dynamics above.
The rest of our proof then follows essentially as that of Theorem \ref{thm:conv-main}.
\end{proof}

\section{Applications to Wireless Networks}
\label{sec:applications}

We now turn to a concrete application of the algorithmic framework presented in the previous sections to distributed throughput maximization in multi-user wireless systems.

\subsection{Problem formulation}
\label{sec:model}

Throughout this section, we will focus on mobile systems where a set $\play = \{1,\dotsc,K\}$ of different \emph{transmitters} (or \emph{users}) communicate simultaneously with a single \emph{receiver} (for instance, a base station or a wireless terminal).
Following recent developments in wireless communication technology \cite{3GPP,3GPP-UE,3GPP-BS}, we will further assume that each user $k\in\play$ is using $\tx_{k}$ antennas for transmission (multiplexing) while the receiver is using $\rx$ antennas for signal reception and decoding.
More precisely, what this means is that the aggregate signal reaching the receiver can be described by the standard channel model
\begin{equation}
\label{eq:signal}
\by
	= \sum_{k=1}^{K} \bH_{k} \bx_{k} + \bz
\end{equation}
where:
\begin{enumerate}
\item
$\by\in\C^{\rx}$ is the aggregate signal reaching the receiver (the channel's \emph{output}).
\item
$\bx_{k} \in \C^{\tx_{k}}$ is the transmitted signal (input) of the $k$-th transmitter (the channel's \emph{input}).
\item
$\bH_{k} \in \C^{\rx\times\tx_{k}}$ denotes the transfer matrix between the $k$-th transmitter and the receiver, representing how the transmit signal is affected by the wireless medium.
To account for channel fading, we will assume in what follows that the users' channel matrices evolve over time following a bounded stationary process \cite{GV97} and we will denote expectations over this distribution by $\ex[\argdot]$.%
\footnote{As a special case, in the static channel regime, we will assume that this process is, in fact, deterministic.}
\item
$\bz\in\C^{\rx}$ is the noise in the channel (including thermal, atmospheric and other peripheral interference effects).
Following standard information-the\-o\-re\-tic caveats, we will further assume that $\bz$ can be modeled as a circularly symmetric, zero-mean Gaussian vector with unit covariance \cite{Tel99,CV93,YRBC04}.
\end{enumerate}

This \ac{MIMO} \ac{MAC} model has attracted considerable interest in the literature \cite{Tel99,CV93,YRBC04,SPB08-jsac,SPB09-sp,BLDH10,PCL03,BLD09} and it is well known that the users' maximum transmission rate is achieved using random Gaussian codes for signal encoding.%
\footnote{Depending on the structure of the channel matrices $\bH_{k}$, the channel model \eqref{eq:signal} actually applies to several telecommunications systems, ranging from \ac{DSL} uplink networks with T{\oe}plitz circulant $\bH_{k}$, to \ac{CDMA}  radio networks \cite{SCS99}.
For concreteness, we will stick here with the interpretation of the signal model \eqref{eq:signal} as an ad hoc multi-user \ac{MIMO} \acl{MAC} with $\bH_{k}$ representing the channel of each link.}
Specifically, let
\begin{equation}
\label{eq:cov}
\bQ_{k}
	= \ex_{\textrm{cb}}[\bx_{k}\bx_{k}^{\dag}]
\end{equation}
denote the covariance matrix of the transmitters' input signal distribution, with the expectation $\ex_{\textrm{cb}}$ being taken over the user's input codebooks.%
\footnote{Importantly, the expectation $\ex_{\textrm{cb}}[\argdot]$ is not related to the expectation $\ex[\argdot]$ taken over the distribution of the users' channels.}
Then, assuming perfect \ac{CSI} at the receiver, the maximum achievable information transmission rate of user $k$ will be given by the familiar expression \cite{Tel99,GV97}:
\begin{equation}
\label{eq:rate}
\rate_{k}(\bQ)
	= \txs
	\ex\left[
	\log\det\left( \bI + \sum_{\ell} \bH_{\ell} \bQ_{\ell} \bH_{\ell}^{\dag} \right)
	- \log\det \bW_{-k}(\bQ_{-k})
	\right],
\end{equation}
where $\bQ = (\bQ_{1},\dotsc,\bQ_{K})$ denotes the users' covariance profile,
$\bQ_{-k}$ is the corresponding profile for all users except $k$,
the expectation $\ex[\argdot]$ is taken over the users' channel law,
and
\begin{equation}
\label{eq:MUI}
\bW_{-k}(\bQ_{-k})
	= \bI + \insum_{\ell\neq k} \bH_{\ell} \bQ_{\ell} \bH_{\ell}^{\dag}
\end{equation}
denotes the \ac{MUI} of user $k$.%
\footnote{From an information-theoretic perspective, we are also assuming \ac{SUD} and perfect \ac{CSI} at the receiver;
for a more detailed account, see e.g. \cite{Tel99,YRBC04,SPB08-jsac,SPB08i-sp,SPB08ii-sp}.}

In this context, the users' objective is to select input signal covariance matrices $\bQ_{k}$ so as to maximize their individual information transmission rate $\rate_{k}(\bQ)$ subject to the constraints
\begin{equation}
\label{eq:power}
\tr(\bQ_{k})
	= \pmax_{k},
\end{equation}
where $\pmax_{k}$ is the transmit power of user $k$ \cite{Tel99,YRBC04}.
More formally, in the language of Section \ref{sec:multi}, the above boils down to the \emph{rate maximization game}:
\begin{equation}
\label{eq:RM}
\tag{RM}
\begin{aligned}
\textrm{maximize unilaterally}
	&\quad
	\pay_{k}(\bX)
	\quad
	\text{for all $k\in\play$},
	\\
\textrm{subject to}
	&\quad
	\bX_{k}\mgeq0,\;
	\tr(\bX_{k}) = 1,
\end{aligned}
\end{equation}
where, for convenience, we have set $\bX_{k} = \bQ_{k}/\pmax_{k}$ and the users' \emph{utility function} $\pay_{k}$ is simply defined as:
\begin{equation}
\pay_{k}(\bX_{1},\dotsc,\bX_{K})
	= \rate_{k}(\pmax_{1}\bX_{1},\dotsc,\pmax_{K}\bX_{K}).
\end{equation}

Clearly, in the presence of fading, the users' objectives are stochastic in nature because of the expectation over $\bH$ in \eqref{eq:rate};
otherwise, in the case of static channels, this expectation is trivial, so \eqref{eq:RM} is deterministic.
As a result, the game-theoretic problem \eqref{eq:RM} can be seen as a special case of the multi-agent stochastic problem \eqref{eq:SSP}.
Indeed, as was shown in \cite{BLDH10}, the users' reward functions $\pay_{k}$ satisfy the potential property \cite{MS96}:
\begin{equation}
\label{eq:pot-rate}
\pay_{k}(\bX_{k};\bX_{-k}) - \pay_{k}(\bX_{k}';\bX_{-k})
	= -\big[ \pot(\bX_{k};\bX_{-k}) - \pot(\bX_{k}';\bX_{-k}) \big]
\end{equation}
where the game's potential function $\pot$ is defined as:
\begin{equation}
\label{eq:pot}
\pot(\bX)
	= \txs
	-\ex\left[ \log\det\left( \bI + \sum_{k} \pmax_{k} \bH_{k} \bX_{k} \bH_{k}^{\dag} \right) \right]
\end{equation}
and the problem's feasible region is $\strat = \prod_{k}\strat_{k}$ with $\strat_{k} = \{\bX_{k} \in \herm_{\tx_{k}} : \bX_{k} \mgeq0 \text{ and } \tr(\bX_{k})=1 \}$.%
\footnote{From an information-theoretic perspective, $\pot$ simply represents  (minus) the users' sum rate under a centralized \ac{SIC} decoding scheme \cite{YRBC04}.
As a result, the above rate maximization problem can be seen both as a game (under \acl{SUD}) or as a distributed, multi-agent optimization problem (under more sophisticated \ac{SIC} schemes).
For a more detailed discussion, see \cite{MBM12,CGM14,BLDH10,MB14,MM15} and references therein.}

Since the function $\bM\mapsto\log\det (\bI + \bM)$ is concave in $\bM$ over the entire cone of positive-semidefinite matrices \cite{BV04}, it follows that $\pot$ is itself convex over $\strat$.%
\footnote{In fact, if the law of the users' channel matrices does not contain any atoms, $\pot$ is actually \emph{strictly} convex.}
Accordingly, the rate maximization game \eqref{eq:RM} falls squarely in the framework of Section \ref{sec:multi}:
in realistic network situations, the distribution of the users' channel matrices is not known to the users, so the rate functions $\rate_{k}$ (or the game's potential $\pot$) cannot be calculated \emph{a priori}.
As a result, to reach a Nash equilibrium of \eqref{eq:RM}, the system's users cannot rely on gradient observations of $\rate_{k}$ (or $\pot$), but only on stochastic (and possibly imperfect and/or delayed) information on the quantities inside the expectation of \eqref{eq:rate} and \eqref{eq:pot}, themselves obtained through an interplay between the transmitters and the receiver.

The framework described above naturally calls for a distributed solution method, so the algorithmic material of Section \ref{sec:analysis} seems particularly well-suited to the occasion.
In the rest of this section, we will describe the specifics of this application.

\subsection{Algorithmic implementation \textendash\ synchronous updates}

The first step required to apply the algorithmic tools of Section \ref{sec:analysis} is to calculate the stochastic gradient of the users' rate functions $\rate_{k}$.
To that end, let
\begin{equation}
r_{k}(\bX)
	= \log\det(\bI + \insum_{\ell} \pmax_{\ell} \bH_{\ell} \bX_{\ell} \bH_{\ell}^{\dag})
	- \log\det(\bI + \insum_{\ell \neq k} \pmax_{\ell} \bH_{\ell} \bX_{\ell} \bH_{\ell}^{\dag}),
\end{equation}
so $\pay_{k}(\bX) = \ex[r_{k}(\bX)]$.
Then, some matrix calculus readily yields
\begin{equation}
\grad_{\bX_{k}} r_{k}(\bX)
	= \bH_{k}^{\dag} \bW^{-1}(\bX;\bH) \bH_{k}^{\dag}, 
\end{equation}
where, in a slight abuse of notation,
\begin{equation}
\bW(\bX;\bH)
	= \txs
	\bI + \sum_{\ell} \pmax_{\ell} \bH_{\ell} \bX_{\ell} \bH_{\ell}^{\dag}
\end{equation}
denotes the aggregate signal covariance matrix at the receiver.
Thus, if $\bH(n)$ denotes the realization of the users' channel matrices at each update period $n=1,2,\dotsc$, and $\bX(n)$ is their corresponding transmit profile, we will assume that
\begin{inparaenum}
[\itshape a\upshape)]
\item
$\bH_{k}(n)$ is measured at each transmitter $k\in\play$;
and
\item
$\bW(\bX(n);\bH(n))$ is measured at the receiver and is then broadcast to the transmitters.
\end{inparaenum}
Under these assumptions, each transmitter $k\in\play$ can recreate their individual (stochastic) gradient matrices at period $n$ as:
\begin{equation}
\label{eq:V-rate}
\hat\bV_{k}(n)
	= \bH_{k}^{\dag}(n) \, \bW^{-1}(\bX(n);\bH(n)) \, \bH_{k}^{\dag}(n),
\end{equation}
and, by construction, we will have:
\begin{equation}
\ex\big[ \hat\bV_{k}(n) \big]
	= \grad_{\bX_{k}(n)} \rate_{k}(\bX(n)),
	\quad
	\text{for all $k=1,\dotsc,K$}.
\end{equation}

With this in mind, Algorithm \ref{alg:DXL} provides the following rate maximization algorithm with \ac{SU}:

\begin{algorithm}[H]
{%
\flushleft
\sf
	\vspace{2pt}
	Parameters:
	discount parameter $\temp>0$;
	decreasing step-size sequence $\step_{n}$.
	\\
	$n \leftarrow 1$;
	\\[2pt]
\ForEach
	{transmitter $k \in \play$}
	{initialize Hermitian score matrix $\bY_{k}\in\herm_{\tx_{k}}$;}
\Repeat
	{$n \leftarrow n+1$;
	\\[2pt]
	Receiver measures and broadcasts $\bP = \bW^{-1}$;
	\\[2pt]
\ForEach
	{transmitter $k \in \play$}
	{\vspace{1ex}
	Measure channel matrix $\bH_{k}$;
	\\[2pt]
	Update score matrix
	$\bY_k \leftarrow \bY_{k} + \step_{n} \left( \bH_{k} \bW^{-1} \bH_{k}^\dag  - \temp\bY_{k} \right)$;
	\\[2pt]
	Update covariance matrix
	$\dis\bX_{k} \leftarrow \frac{\exp(\bY_{k})}{\tr[\exp(\bY_{k})]}$.
	} 
\textbf{until} termination criterion is reached.
} 
} 
\caption{\ac{MIMO} rate maximization with synchronous updates (\ac{SU})}
\label{alg:DXL-SU}
\end{algorithm}

From the point of view of distributed implementation, Algorithm \ref{alg:DXL-SU} has the following desirable properties:
\begin{enumerate}[(P1)]
\item
It is \emph{distributed}: users only update their individual variables using the same information as in distributed water-filling (namely the broadcast of $\bW^{-1}$) \cite{YRBC04,SPB08-jsac,SPB08i-sp,SPB08ii-sp}.
\item
It is \emph{stateless}:
users do not need to know the state of the system (or the existence of other users).
\item
It is \emph{reinforcing}:
users tend to increase their individual transmission rates $\pay_{k}$.
\item
It is \emph{stable}:
the matrix exponentials can be calculated in a numerically stable and efficient manner \cite{Moler2003}.
\end{enumerate}
Furthermore, since the users' channels are bounded by necessity, Theorem \ref{thm:conv-main} readily yields:

\begin{corollary}
\label{cor:conv-rate-SU}
Assume that Algorithm \ref{alg:DXL-SU} is run with a step-size sequence $\step_{n}$ such that $\sum_{n=1}^{\infty} \step_{n}^{2} < \sum_{n=1}^{\infty} \step_{n} = \infty$.
Then, the algorithm's iterates converge \textup(a.s.\textup) within $\eps(\temp)$ of a Nash equilibrium of the rate maximization game \eqref{eq:RM} and the approximation error $\eps(\temp)$ vanishes as $\temp\to0^{+}$.
\end{corollary}

\subsection{Asynchronous implementation}
\label{sec:async}

Let us now consider a more realistic wireless environement where the transmitters do not share a common update clock \textendash\ so synchronous decisions are not possible.
In this context, the synchronous update structure of Algorithm \ref{alg:DXL-SU} is no longer appropriate, so we will employ Algorithm \ref{alg:DXL-multi} (which is fully decentralized) instead.

To that end, assume that each transmitter is equipped with an individual timer $\tau_{k}$ whose ticks indicate the update events of user $k$.
More precisely, we assume here that $\tau_{k}\from\N\to\R_{+}$ is an increasing (and possibly random) sequence such that $\tau_{k}(n)$ marks the instance at which the $k$-th user updates his covariance matrix $\bX_{k}$ for the $n$-th time \textendash\ so $\bX_{k}$ does not change between $\tau_{k}(n)$ and  $\tau_{k}(n+1)$.
Similarly, we assume that the receiver is equipped with a timer $\tau_{0}(n)$ that triggers the measurements of
\begin{equation}
\bW(t)
	\equiv \bW(\bX(t);\bH(t))
	= \bI + \insum_{\ell} \pmax_{\ell} \bH_{\ell}(t) \bX_{\ell}(t) \bH_{\ell}^{\dag}(t).
\end{equation}
Thus,
at every tick of $\tau_{k}$, user $k$ measures $\bH_{k}$ and updates $\bX_{k}$
while,
at every tick of $\tau_{0}$, the receiver measures and broadcasts $\bW$.
This asynchronous operating mode fits naturally within the framework of Algorithm \ref{alg:DXL-multi}, leading in turn to the following implementation of Algorithm \ref{alg:DXL-SU} with \ac{AU}:

\begin{algorithm}[H]
{%
\flushleft
\sf
\vspace{2pt}
	Parameters:
	discount parameter $\temp>0$;
	\\[2pt]
	$n \leftarrow 1$;
	\\[2pt]
	Initialize Hermitian score matrix $\bY$.
	\\[2pt]
\Repeat
{%
\texttt{UpdateEvent} occurs at time $\tau(n)$;
	\\[2pt]
	$n \leftarrow n+1$;
	\\[2pt]
	Measure channel matrix $\bH$;
	\\[2pt]
	Recall latest broadcast of $\bW$;
	\\[2pt]
	Update score matrix
	$\bY \leftarrow \bY + \frac{1}{n} \left( \bH \bW^{-1} \bH^{\dag}  - \temp\bY \right)$;
	\\[2pt]
	Update covariance matrix
	$\dis\bX \leftarrow \frac{\exp(\bY)}{\tr[\exp(\bY)]}$;
	\\[2pt]
\textbf{until} termination criterion is reached.
} 
} 
\caption{\ac{MIMO} rate maximization with asynchronous updates (\ac{AU})}
\label{alg:DXL-AU}
\end{algorithm}

Algorithm \ref{alg:DXL-AU} is run independently by each transmitter \textendash\ though, of course, if all transmitters share a common timer, Algorithm \ref{alg:DXL-AU} reduces to the synchronous context of Algorithm \ref{alg:DXL-SU}.
Moreover, provided that all individual timers $\tau_{k}$ have positive finite rate (i.e. $\lim \tau_{k}(n)/n$ exists and is finite), it is easy to see that the update sequence generated by Algorithm \ref{alg:DXL-AU} satisfies the assumptions of Theorem \ref{thm:conv-multi}.
Indeed, the set-valued process $\play_{n}$ used in \eqref{eq:DXL-multi} to indicate the set of transmitters updating their covariance matrices at the $n$-th overall update event may be obtained from the users' individual timers $\tau_{k}$ as follows:
First, let $\play(t) = \{k\in\play: \tau_{k}(n) = t \text{ for some $n\in\N$}\}$ denote the set of players updating at time $t$ and let $n(t) = \card\{s\leq t: \play(s) \neq \varnothing\}$ be the total number of update epochs up to time $t$.
Then, $\play_{n} = \play(\inf\{t: n(t) \geq n\})$ and $n_{k}(n) = \sum_{r=1}^{n} \one(k \in \play_{r})$, so the limit $\lim \tau_{k}(n)/n$ exists and is finite if and only if the limit $\lim n_{k}(n)/n$ exists and is positive.
With all this in mind, we readily obtain:

\begin{corollary}
\label{cor:conv-rate-AU}
The iterates of Algorithm \ref{alg:DXL-AU} converge \textup(a.s.\textup) within $\eps(\temp)$ of a Nash equilibrium of the rate maximization game \eqref{eq:RM} and the approximation error $\eps(\temp)$ vanishes as $\temp\to0^{+}$.
\end{corollary}

\subsection{Learning with imperfect information}
\label{sec:imperfect}

As a final application of the algorithmic framework of Section \ref{sec:analysis} to the problem at hand, we turn to the case where the users' channel matrices $\bH_{k}$ are static but channel and interference measurements are subject to observation noise and measurement errors.
In this case, the rate maximization game \eqref{eq:RM} becomes deterministic but the system is still subject to stochasticity originating from noise and uncertainty in the users' measurements.

In the perfect information case, the (deterministic) semidefinite problem \eqref{eq:RM} may be solved by water-filling techniques \cite{CV93}, properly adapted to multi-user environments \cite{YRBC04,SPB08-jsac,SPB06}.
Such methods can be either \emph{iterative} (with users updating their covariance matrices one after the other, in a round-robin fashion) \cite{YRBC04} or \emph{simultaneous} (with users updating all at once) \cite{SPB06}.
The benefit of the former (iterative) scheme is that its convergence is guaranteed \cite{YRBC04};
however, the algorithm's convergence rate is inversely proportional to the number of users in the system (making such methods unsuitable for large networks).
On the other hand, simultaneous water-filling methods are faster \cite{SPB06}, but their convergence is conditional on certain ``mild intereference'' conditions which fail to hold even in very simple $2\times2$ systems \cite{MBML12}.
Making matters worse, water-filling methods rely on perfect \ac{CSIT} and perfect measurements of the output signal covariance matrix $\bW$ at the receiver;
when it is impossible (or impractical) to obtain such noiseless measurements, it is not known whether water-filling methods converge.


In light of the above, our goal here will be to provide a viable alternative to water-filling based on \eqref{eq:DXL}.
To that end, we will focus on two sources of measurement noise:
\begin{enumerate}
\item
The (static) transfer matrices $\bH_k$ can only be measured at the transmitter up to some random observational error.
\item
The receiver can only estimate the covariance $\bW$ of the aggregate received signal $\by$ via random sampling (assumed to occur between the updates of the transmitters).
\end{enumerate}
Even though these two randomness sources are independent of one another, the gradient matrices $\bV_{k} = \bH_{k} \bW^{-1} \bH_{k}^{\dag}$ of \eqref{eq:V-rate} depend nonlinearly on $\bH_{k}$ and $\bW$, so care must be taken to construct an unbiased estimator of $\bV_{k}$ from noisy estimates of $\bH_{k}$ and $\bW$.

We first consider the random perturbations induced on the estimation of $\bW^{-1}$ by signal sampling at the receiver end.
On that account, recall that $\bW$ is simply the covariance matrix of the aggregate received signal $\by\in\C^{\rx}$:
\begin{flalign}
\ex[\by \by^{\dag}]
	=\ex\big[\bz \bz^{\dag} \big] + \insum_{k} \bH_{k} \ex\big[\bx_{k} \bx_{k}^{\dag} \big] \bH_{k}^{\dag}
	= \bI + \insum_{k} \bH_{k} \bQ_{k} \bH_{k}^{\dag}
	= \bW.
\end{flalign}
As a result, an unbiased estimate for the covariance $\bW$ of $\by$ may be obtained from a systematically unbiased sample $\by_{1},\dotsc, \by_{S}$ of $\by$ by means of the classical estimator $\hat \bW = S^{-1} \insum_{s=1}^{S} \by_{s} \by_{s}^{\dag}$.%
\footnote{Since $\ex_{\textrm{cb}}[\by] = 0$, we do not need to include an $S/(S-1)$ bias correction factor in the estimate of $\bW$.
Also, in a slight abuse of notation, the measurement expectations here are taken with respect to the law of $\bx$, $\by$, and $\bz$.}

On the other hand, given that $\hat\bW^{-1}$ is a \emph{biased} estimator of $\bW^{-1}$ (and hence introduces a systematic error to the measurement process) \cite{And03}, we cannot use this classical covariance estimate for the received signal precision (inverse covariance) matrix $\bW^{-1}$.
Instead, following \cite{And03}, an unbiased estimate of the precision matrix $\bP = \bW^{-1}$ of $\by$ is given by the corrected expression:
\begin{equation}
\label{eq:Winv-unbiased}
\hat\bP
	= \frac{S - \rx -1}{S} \hat\bW^{-1},
\end{equation}
where $\hat\bW = S^{-1} \insum_{s=1}^{S} \by_{s} \by_{s}^{\dag}$ as before.
Thus, to obtain $\bW^{-1}$, the receiver only needs to take $S > \rx + 1$ independent measurements of $\by$ and then broadcast the unbiased estimate $\hat\bP$ of $\bW^{-1}$ to the network's users.

Similarly, in the absence of perfect \acl{CSIT}, the users must obtain an unbiased estimate of the unilateral gradient matrices $\bV_{k} = \bH_{k}^{\dag} \bW^{-1} \bH_{k}$ from the broadcasted value of $\bW$ and imperfect measurements of their channel matrices $\bH_{k}$.
However, an added complication here is that the estimated matrix $\hat\bV_{k}$ must be itself Hermitian \textendash\
otherwise, $\bQ_{k}$ need not be positive-definite and the \ac{DXL} scheme may fail to be well-posed.
To accommodate this requirement, if each transmitter takes $S>1$ independent measurements $\hat\bH_{k,1},\dotsc,\hat\bH_{k,S}$ of their individual channel matrix $\bH_{k}$, such an estimator is given by the expression:
\begin{equation}
\label{eq:V-unbiased}
\hat\bV_{k}
	= \frac{1}{S(S-1)}
	\insum_{s\neq s'} \hat\bH_{k,s}^{\dag} \hat \bP \hat\bH_{k,s'},
\end{equation}
where $\hat\bP$ is the broadcast estimate \eqref{eq:Winv-unbiased} of $\bW^{-1}$.
Indeed, given that the sample measurements $\bH_{k,s}$ are assumed stochastically independent, we will have:
\begin{flalign}
\ex[\hat\bV_{k}]
	&= \frac{1}{S(S-1)}
	\insum_{s\neq s'} \ex\big[ \bH_{k,s}^{\dag} \hat \bP \bH_{k,s'}\big]
	\notag\\
	&= \frac{1}{S(S-1)} \insum_{s\neq s'} \ex[\hat\bH_{k,s}^{\dag}] \ex[\hat \bP] \ex[\hat \bH_{k,s}]
	= \bH_{k}^{\dag} \bW^{-1} \bH_{k},
\end{flalign}
where we have used the independence of the samples to decorrelate the expectations in the second equality, and we relied on the unbiasedness of $\hat\bP$ and $\hat\bH_{k}$ for the last one.
Thus, with $\ex[\hat\bV_{k}] = \bV_{k}$, our construction of an unbiased estimator for $\bV_{k}$ is complete.

From an implementation viewpoint, the above leads to the following distributed operation protocol.
First, with notation as in the previous section, let $\tau_{0}$ denote the receiver's measurement timer (so $\tau_{0}(n)$ is the $n$-th instance in time at which the receiver measures and broadcast $\bW^{-1}$).
Then, at each tick of $\tau_{0}$, the receiver takes a sample of the received signal $\by$ of size $S > \tx + 1$ and computes the estimate $\hat\bP(\tau_{0}(n))$ as above.%
\footnote{We implicitly assume here that this measurement process takes a negligible amount of time.
This assumption is justified by the fact that the characteristic time at which the receiver estimates $\bW$ for decoding purposes is much shorter than the interval between user updates.}
Likewise, if $\tau_{k}$ is the update timer of user $k$ (so $\tau_{k}(n)$ is the $n$-th update time for user $k$), each transmitter $k\in\play$ is assumed to measure his individual channel matrix and calculate his gradient estimate $\hat\bV_{k}$ using the recipe \eqref{eq:V-unbiased} with the latest broadcasted value of $\hat\bP$.%
\footnote{Obviously, if $\tau_{0} = \tau_{k}$ for all $k\in\play$, the above process boils down to the synchronous regime of Algorithm \ref{alg:DXL-SU}.}
Theorems \ref{thm:conv-main} and \ref{thm:conv-multi} then yield:

\begin{corollary}
\label{cor:conv-rate-noisy}
With notation as before, the iterates of Algorithm \ref{alg:DXL-AU} with imperfect feedback converge \textup(a.s.\textup) within $\eps(\temp)$ of a Nash equilibrium of the rate maximization game \eqref{eq:RM} with static channels;
moreover, the approximation error $\eps(\temp)$ vanishes as $\temp\to0^{+}$.
\end{corollary}

\subsection{Numerical results}
\label{sec:numerics}

To assess the performance of \eqref{eq:DXL} applied to realistic network conditions, we simulated in Fig.~\ref{fig:discount} a multi-user uplink \ac{MIMO} system consisting of a wireless base receiver with $5$ antennas and $K = 25$ transmitters, each with a random number $m_{k}$ of transmit antennas picked uniformly between $2$ and $6$.%
\footnote{For simplicity, throughout our numerical simulations, we focused on the \acl{SU} case (Algorithm \ref{alg:DXL-SU}).}
For the static channel case, each user's channel matrix $\bH_{k}$ was drawn from a complex Gaussian distribution at the outset of the transmission (but remained static once chosen), and Algorithm \ref{alg:DXL-SU} was ran with a large constant step size for different values of the discount parameter $\temp$.
The performance of the algorithm over time was then assessed by plotting the normalized efficiency ratio
\begin{equation}
\label{eq:efficiency}
\eff(n)
	= \frac{\pot_{\max} - \pot_{n}}{\pot_{\max} - \pot_{\min}},
\end{equation}
where $\pot_{n}$ denotes the users' sum rate at the $n$-th iteration of the algorithm, and $\pot_{\max}$ (resp. $\pot_{\min}$) is the maximum (resp. minimum) value of $\pot$ over the system's set $\strat$ of feasible covariance matrices.
Thus, by definition, an efficiency measure of $1$ corresponds to a Nash equilibrium of the rate maximization game \eqref{eq:RM} while an efficiency ratio of $0$ means that the system is very far from equilibrium.%
\footnote{The reason for using this efficiency measure instead of the user's sum rate $\pot$ directly, was to eliminate any scaling artifacts arising e.g. from $\pot$ taking values in a very narrow band close to its maximum value.}
In tune with Theorem \ref{thm:conv-main}, Fig.~\ref{fig:discount} reveals that Algorithm \ref{alg:DXL-SU} converges within a few iterations (effectively, within a single iteration for low $\temp$), but the end value of the users' sum rate deteriorates for higher values of the discount parameter $\temp$.

\begin{figure}[t]
\centering
\small
\includegraphics[width=.8\textwidth]{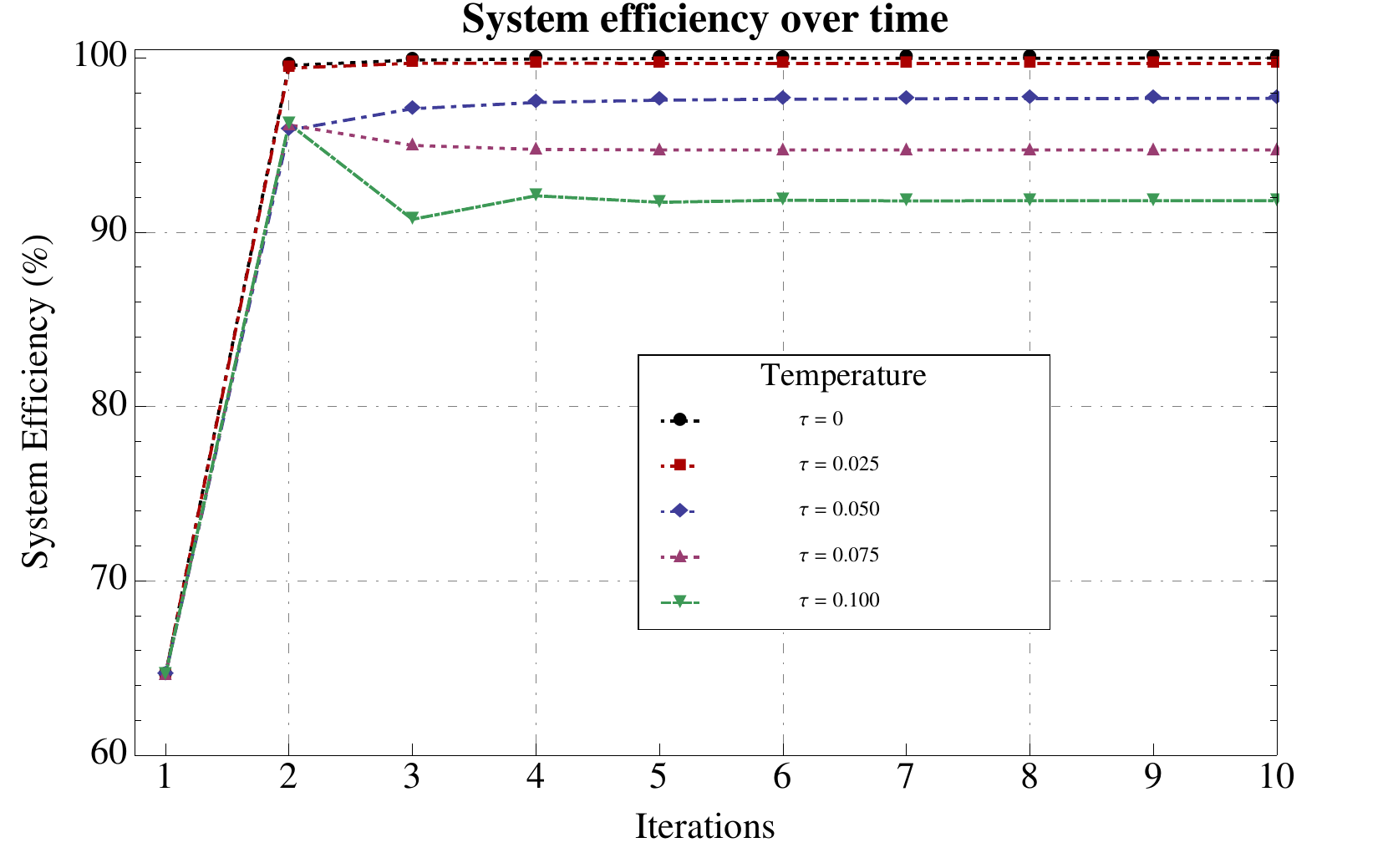}
\caption{The effect of the discount parameter $\temp$ on the end-state of the \acl{DXL} algorithm (Algorithm \ref{alg:DXL-SU}).}
\label{fig:discount}
\end{figure}

In Fig.~\ref{fig:convspeed}, we fix the algorithm's discount parameter to a low level ($\temp = 10^{-3}$) that ensures effective convergence to Nash equilibrium, and we investigate the algorithm's convergence speed as a function of the number of transmitters, using existing water-filling methods as a benchmark.
Specifically, in Fig.~\ref{fig:users}, we ran Algorithm \ref{alg:DXL-SU} for a multi-user uplink \ac{MIMO} system with $K = 10$, $25$, $50$ and $100$ users using a large, constant step size;
as a result of this parameter tuning, Algorithm \ref{alg:DXL-SU} effectively attains the system's sum capacity within one or two iterations, even for large numbers of users.
Importantly, as can be seen in Fig.~\ref{fig:WF}, this represents a marked improvement over water-filling methods, even in moderately-sized systems with $K = 25$ users:
on the one hand, \ac{IWF} \cite{YRBC04} is significantly slower than \ac{SU} (it requires $\bigoh(K)$ iterations to achieve the same performance level as the first iteration of Algorithm \ref{alg:DXL-SU}), whereas \ac{SWF} \cite{Scu06} fails to converge altogether.

\begin{figure*}[t]
\centering
\small
\subfigure
[Convergence speed of Algorithm \ref{alg:DXL-SU} for different numbers of users.]
{\label{fig:users}
\includegraphics[width=.48\textwidth]{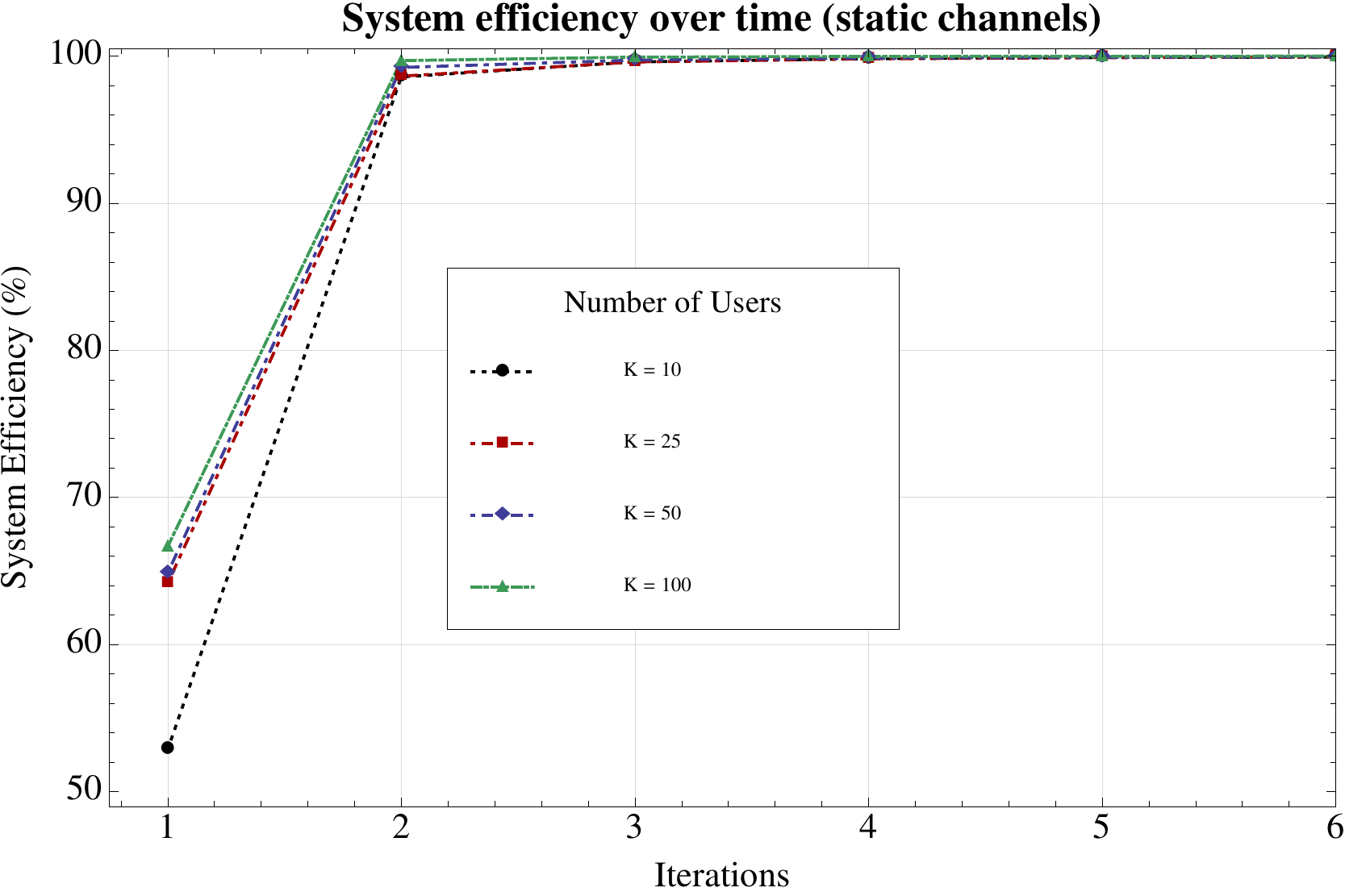}}
\hfill
\subfigure
[Discounted exponential learning vs. water-filling for $K=25$ users.]
{\label{fig:WF}
\includegraphics[width=.48\textwidth]{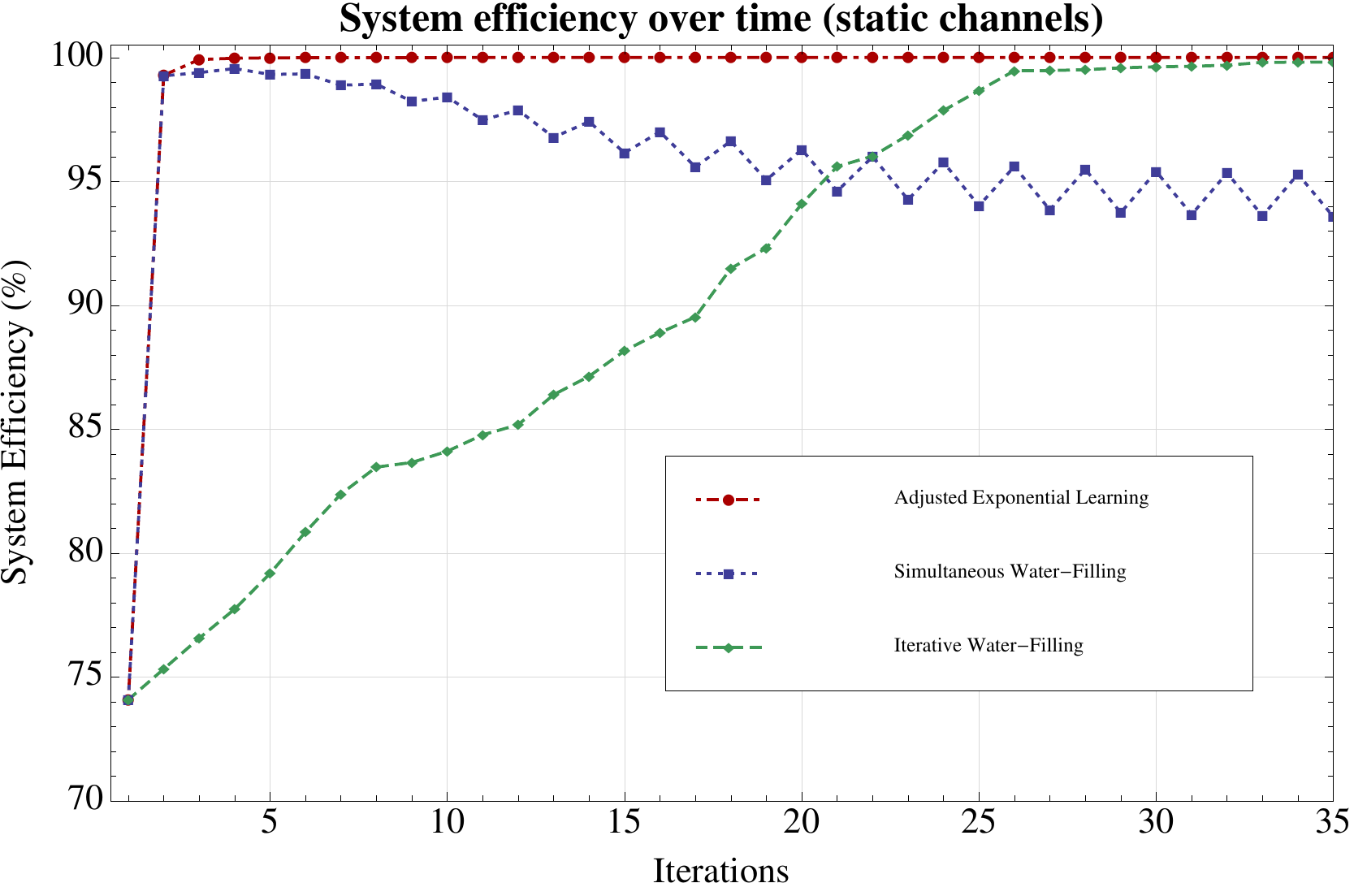}}%
\caption{The convergence speed of \acl{DXL} with \aclp{SU} as a function of the number of users.}
\label{fig:convspeed}
\end{figure*}

The robustness of \acl{DXL} is investigated further in Fig.~\ref{fig:noise} where we simulate an uplink \ac{MIMO} system consisting of $K=25$ transmitters with imperfect \ac{CSI} and noisy measurements at the receiver.
For simplicity, we modeled these errors as additive \acs{iid} zero-mean Gaussian perturbations to the matrices $\bV_{k} = \bH_{k} \bW^{-1} \bH_{k}^{\dag}$ that are used in the update step of \ac{SU}, and the strength of these perturbations was controlled by the ratio of the errors' standard deviation to the matrix norm of $\bV_{k}$ (so a relative error level of $\eta = 100\%$ means that the measurement error has the same magnitude as the measured variable).
We then plotted the efficiency ratio achieved by Algorithm \ref{alg:DXL-SU} over time for average error levels of $\eta = 15\%$ and $\eta = 100\%$;
for benchmarking purposes, we then also ran the iterative and simultaneous water-filling algorithms with the same relative error levels (and noise realizations).
As can be seen in Fig.~\ref{fig:noise}, the performance of water-filling methods remains acceptable at low error levels (attaining 90\textendash95\% of the system's sum capacity), but when the measurement noise gets higher, water-filling offers no perceptible advantage over the users' initial choice of covariance matrices.
By contrast, \acl{DXL} retains its convergence properties even for relative error levels as high as $100\%$ \textendash\ though, of course, the algorithm's convergence speed is negatively impacted.

\begin{figure*}[t]
\centering
\small
\subfigure
[Learning with a relative error level of 15\%.]
{\label{fig:noise.low}
\includegraphics[width=.48\textwidth]{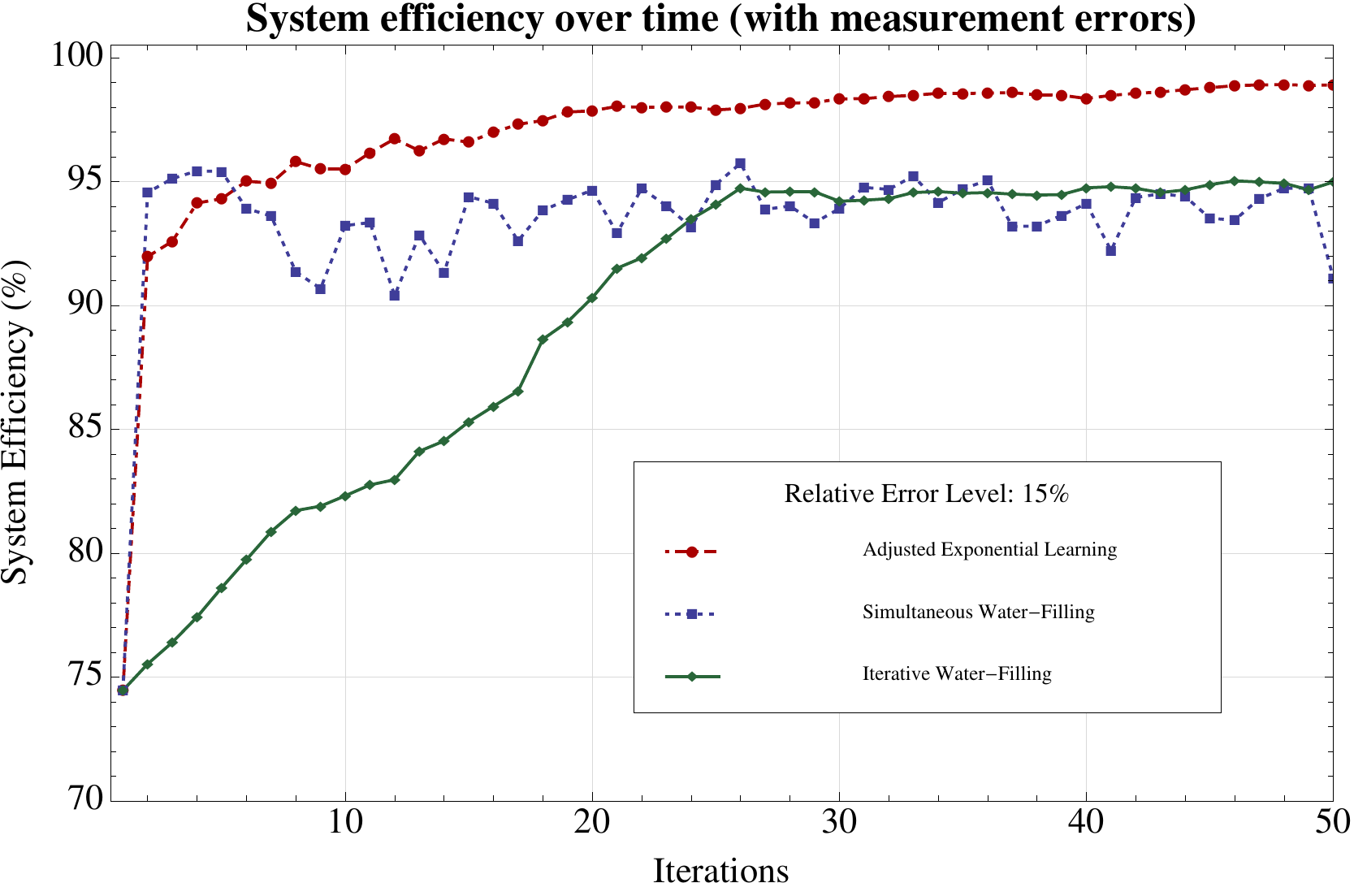}}
\hfill
\subfigure
[Learning with a relative error level of 100\%.]
{\label{fig:noise.high}
\includegraphics[width=.48\textwidth]{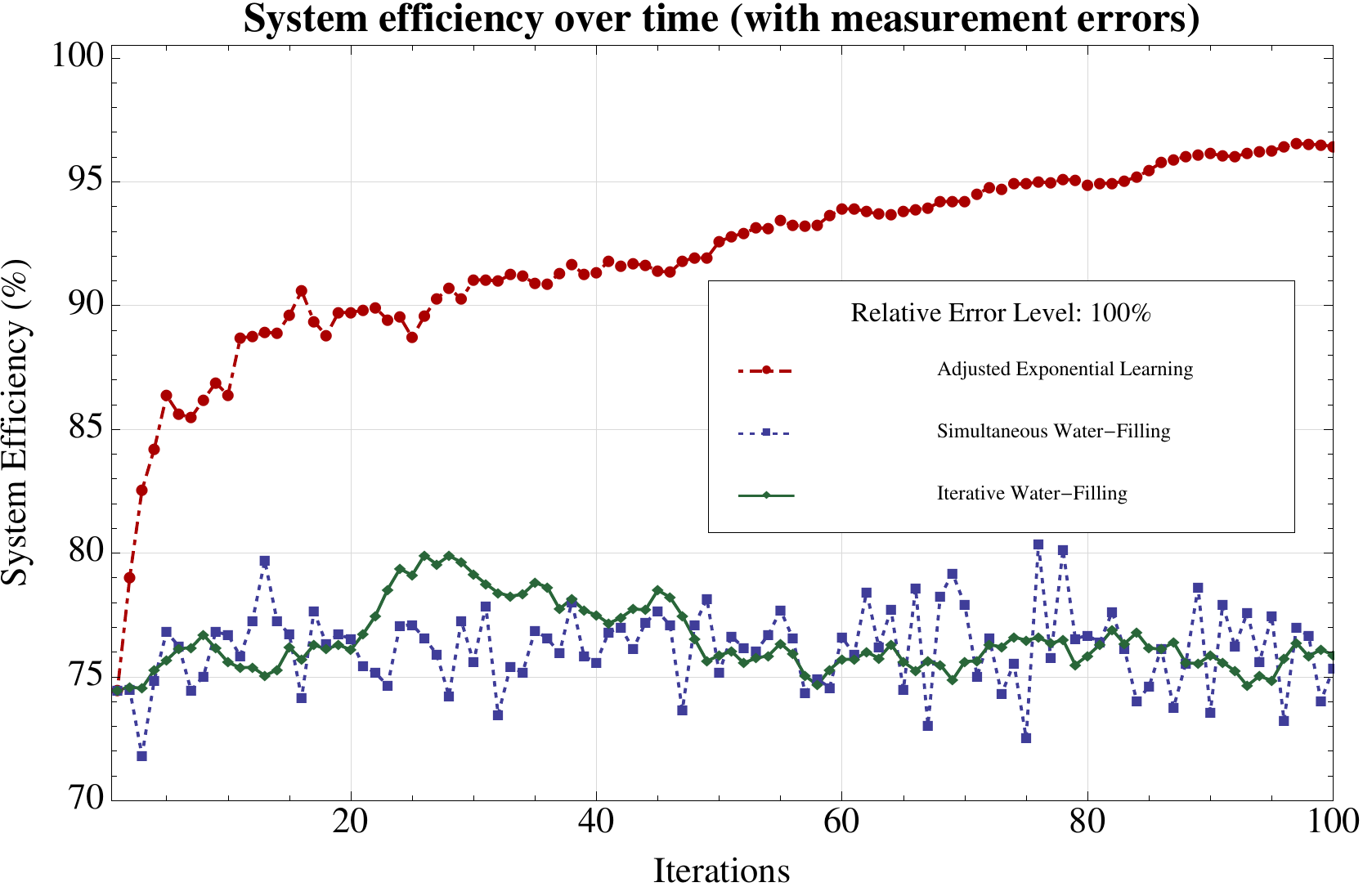}}%
\caption{The robustness of entropy-driven learning in the presence of measurement errors:
in contrast to water-filling methods, the entropy-driven learning attains the channel's sum capacity, even in the presence of very high measurement errors.}
\label{fig:noise}
\end{figure*}

Finally, to account for changing channel conditions, we also plotted the performance of Algorithm \ref{alg:DXL-SU} for non-static channels following the well-known Jakes model of Rayleigh fading \cite{Cal++07}.
More precisely, in Fig.~\ref{fig:tracking}, we consider a \ac{MIMO} uplink system with $3$ receive antennas and $K = 10$ users with $2$ antennas each, transmitting at a frequency of $f = 2\,\mathrm{GHz}$ and with average pedestrian velocities of $v = 5\,\mathrm{km/h}$ (corresponding to a channel coherence time of $108$ ms).
We then ran Algorithm \ref{alg:DXL-SU} with an update period of $\delta = 3\,\mathrm{ms}$, and we plotted the achieved sum rate $\pot(t)$ at time $t$ versus the maximum attainable sum rate $\pot_{\max}(t)$ given the channel matrices $\bH_{k}(t)$ at time $t$ (and versus the ``uniform'' sum rate that users could achieve by spreading their power uniformly over their antennas).
As a result of its high convergence speed, Algorithm \ref{alg:DXL-SU} tracks the system's sum capacity remarkably well, despite the changing channel conditions.
Moreover, the sum rate difference between the learned transmit covariance profile and the uniform one shows that this tracking is not an artifact of the system's sum capacity always being within a narrow band of its (evolving) maximum, but a real consequence of learning.

\begin{figure}[t]
\centering
\small
\includegraphics[width=.8\textwidth]{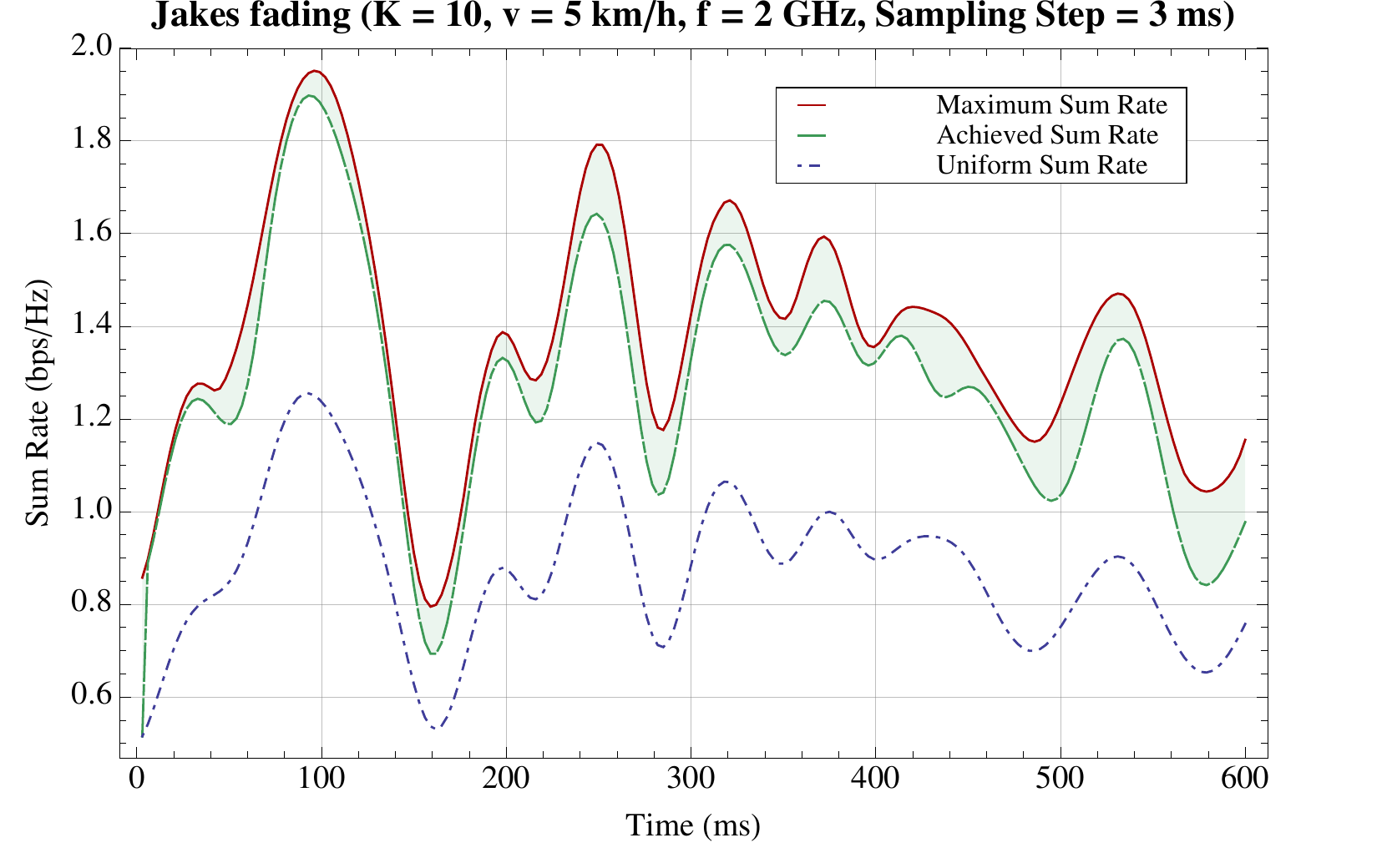}
\caption{The performance of entropy-driven learning under changing channel conditions (following the Jakes model for Rayleigh fading with parameters indicated in the figure caption).}
\label{fig:tracking}
\end{figure}

\section{Conclusions}
\label{sec.conclusions}

In this paper, we introduced a class of distributed algorithms based
on a regularized variant of matrix exponential learning for
stochastic semidefinite programming with applications to robust spectrum management in multi-user \ac{MIMO} systems.
This  adjustment  of classical exponential learning generates a discrete-time algorithm
which tracks the continuous-time dynamics of adjusted exponential
learning and converges arbitrarily close to the system's optimum configuration.
Thanks to this adjustment term, the algorithm remains robust in the presence of stochastic perturbations:
it converges even when the agents only have imperfect (or delayed)
information  at their disposal, or even if they update in a fully asynchronous manner and independently of one another.

The optimization method of adjusted exponential learning method actually applies to a wide range of semidefinite problems;
we focused here on the \ac{MIMO} \ac{MAC} where our approach dominates
classical water filling techniquesboth in terms of speed of
convergence and robustness to random perturbations.

In the case of multi-user MIMO systems, it out-performs  traditional
water-filling methods, both in terms of robustness to imperfect signal
measurements and speed of convergence: in
practice, the algorithme converges within a few iterations, even for
large numbers of antennas.


\bibliographystyle{siam}
\footnotesize
\setlength{\bibsep}{0pt}
\bibliography{IEEEabrv,Bibliography}

\end{document}